\documentclass[aop,preprint]{imsart}

\RequirePackage[OT1]{fontenc}
\RequirePackage{amsthm,amsmath}
\RequirePackage[numbers]{natbib}
\RequirePackage[colorlinks,citecolor=blue,urlcolor=blue]{hyperref}

\usepackage{amsmath, latexsym, amsfonts, amssymb, amsthm}
\usepackage{graphicx, color,hyperref,dsfont,epsfig,caption,wrapfig,subfig}
\usepackage{pstricks}

\arxiv{arXiv:1804.02942}

\startlocaldefs
\numberwithin{equation}{section}
\theoremstyle{plain}

\endlocaldefs

\numberwithin{equation}{section}

\newtheorem{theorem}{Theorem}[section]
\newtheorem{lemma}[theorem]{Lemma}
\newtheorem{proposition}[theorem]{Proposition}
\newtheorem{corollary}[theorem]{Corollary}

\newtheorem{remark}[theorem]{Remark}
\newtheorem{definition}[theorem]{Definition}

\renewcommand{\P}{\mathbb{P}}
\newcommand{\E}{\mathbb{E}}

\begin{document}

\begin{frontmatter}
\title{The distribution of Gaussian multiplicative chaos on the unit interval}
\runtitle{The distribution of GMC on the interval}

\begin{aug}
\author{\fnms{Guillaume} \snm{Remy}\thanksref{m1}\ead[label=e1]{remy@math.columbia.edu}}
\and
\author{\fnms{Tunan} \snm{Zhu}\thanksref{m2}
\ead[label=e2]{tunan.zhu@ens.fr}}

\thankstext{m1}{$^{,\ddagger}$Research supported in part by the ANR grant Liouville (ANR-15-CE40-0013).}
\runauthor{G. Remy and T. Zhu}

\affiliation{Department of Mathematics, Columbia University.\thanksmark{m1}}
\affiliation{D\'epartement de Math\'ematiques et Applications,\\ \'Ecole Normale Sup\'erieure de Paris.\thanksmark{m2}}

\address{Guillaume Remy\\
Department of Mathematics,\\
Columbia University,\\
2990 Broadway, New York, NY 10027, USA.\\
\printead{e1}}

\address{Tunan Zhu\\
D\'epartement de math\'ematiques et applications,\\
\'Ecole normale sup\'erieure de Paris,\\
45 Rue d'Ulm, 75005 Paris, France.\\
\printead{e2}}

\end{aug}

\begin{abstract}
We consider a sub-critical Gaussian multiplicative chaos (GMC) measure defined on the unit interval $[0,1]$ and prove an exact formula for the fractional moments of the total mass of this measure. Our formula includes the case where log-singularities (also called insertion points) are added in $0$ and $1$, the most general case predicted by the Selberg integral. The idea to perform this computation is to introduce certain auxiliary functions resembling holomorphic observables of conformal field theory that will be solutions of hypergeometric equations. Solving these equations then provides non-trivial relations that completely determine the moments we wish to compute. We also include a detailed discussion of the so-called reflection coefficients appearing in tail expansions of GMC measures and in Liouville theory. Our theorem provides an exact value for one of these coefficients. Lastly we mention some additional applications to small deviations for GMC measures, to the behavior of the maximum of the log-correlated field on the interval and to random hermitian matrices.
\end{abstract}



\end{frontmatter}

\section{Introduction and main result}

Starting from a log-correlated field $X$ one can define the associated Gaussian multiplicative chaos (GMC) measure which has a density with respect to the Lebesgue measure formally given by the exponential of $X$. This definition is formal as $X$ lives in the space of distributions but since the pioneering work of Kahane \cite{Kah} in 1985 it is well understood how to give a rigorous probabilistic definition to these GMC measures by using a limiting procedure. Ever since GMC has been extensively studied in probability theory and mathematical physics with applications including 3d turbulence, statistical physics, mathematical finance, random geometry and 2d quantum gravity. See for instance \cite{review} for a review.

Despite the importance of GMC measures in many active fields of research, rigorous computations have remained until very recently completely out of reach. A large number of exact formulas have been conjectured by the physicists' trick of analytic continuation from positive integers to real numbers (see the explanations below) but with no indication of how to rigorously prove such formulas. A decisive step was made in \cite{Sphere} where a connection is uncovered between GMC measures and the correlation functions of Liouville conformal field theory (LCFT). By implementing the techniques of conformal field theory (CFT) in a probabilistic setting one can hope to perform rigorous computations on GMC.

Indeed, in 2017 a proof was given by Kupiainen-Rhodes-Vargas of the celebrated DOZZ formula \cite{DOZZ1, DOZZ2} first conjectured independently by Dorn and Otto in \cite{DO} and by Zamolodchikov and Zamolodchikov in \cite{ZZ}. This formula gives the value of the three-point correlation function of LCFT on the Riemann sphere and it can also be seen as the first rigorous computation of fractional moments of a GMC measure. Very shortly after, the study of LCFT on the unit disk by the first author led in \cite{remy} to the proof of a probability density for the total mass of the GMC measure on the unit circle. This result proves the conjecture of Fyodorov and Bouchaud stated in \cite{FyBo} and it is the first explicit probability density for a GMC measure obtained in the mathematical literature. 

The present paper presents a third case where exact computations are tractable using CFT-inspired techniques which is the case of GMC on the unit interval $[0,1]$ with $X$ of covariance written below \eqref{covariance}. This model was studied by Bacry-Muzy in \cite{Bacry} where they prove existence of moments and other properties of GMC. Five years after exact formulas for this model on the interval were conjectured independently by Fyodorov-Le Doussal-Rosso in  \cite{FLeR, FLe} and by Ostrovsky in  \cite{Ostro5, Ostro1}. In \cite{FLeR, FLe} the exact formulas are found using an analytic continuation from integers to real numbers but in his papers Ostrovsky went a step further and showed that the formulas did correspond to a valid probability distribution. He also performs the computation of the derivatives of all order in $\gamma$ of \eqref{main_quantity} at $\gamma=0$ which is referred to as the intermittency differentiation. However a crucial analycity argument is missing for this approach to prove rigorously an exact formula. See \cite{Ostro_review} for a beautiful review on all the known results and conjectures for the GMC on the interval (and also for the similar model on the circle) as well as for many additional references. 

The main result of our work is precisely the proof of these conjectures for the GMC measure on $[0,1]$. The major input of our paper is the introduction of two auxiliary functions that will be solutions to hypergeometric equations, see Proposition \ref{BPZ}. This observation was to the best of our knowledge unknown to the statistical physics community although an analogous statement was known in the case of the Selberg integral, see \cite{kaneko} and the explanations of subsection \ref{strategy}. By studying the solution space of these differential equations we obtain non-trivial relations on the GMC that allow us to rigorously prove the formulas conjectured by physicists.

Let us now introduce the framework of our paper. We consider the log-correlated field $X$ on the interval $[0,1]$ with covariance given for $x,y \in [0,1]$ by:
\begin{equation}\label{covariance}
\mathbb{E}[ X(x) X(y) ] = 2 \ln \frac{1}{\vert x - y \vert}.\footnote{Our normalization differs from the $\ln \frac{1}{\vert x - y \vert}$ usually found in the literature.}
\end{equation}
Because of the singularity of its covariance $X$ is not defined pointwise and lives in the space of distributions. We define the associated GMC measure on the interval $[0,1]$ by the standard regularization procedure for $ \gamma \in (0,2)$,
\begin{equation}\label{def_GMC}
e^{\frac{\gamma}{2} X(x) } dx := \lim_{\delta  \rightarrow 0} e^{\frac{\gamma}{2} X_{\delta}(x) - \frac{\gamma^2}{8}\mathbb{E}[X_{\delta}(x)^2 ] } dx,
\end{equation}
where $X_{\delta}$ stands for any reasonable cut-off of $X$ that converges to $X$ as $\delta $ goes to $0$. The convergence in \eqref{def_GMC} is in probability with respect to the weak topology of measures, meaning that for all continuous test functions $f: [0,1] \mapsto \mathbb{R} $ the following holds in probability:
\begin{equation}
\int_0^1  f(x) e^{\frac{\gamma}{2} X(x) } dx = \lim_{\delta  \rightarrow 0} \int_0^1 f(x) e^{\frac{\gamma}{2} X_{\delta}(x) - \frac{\gamma^2}{8}\mathbb{E}[X_{\delta}(x)^2 ] } dx.
\end{equation}
For an elementary proof of this convergence see \cite{Ber}. We now introduce the main quantity of interest of our paper, for $\gamma \in (0,2)$ and for real $p$, $a$, $b$:
\begin{equation}\label{main_quantity}
M(\gamma, p, a,b )  :=  \mathbb{E}[ (\int_{0}^{1}  x^{a}(1-x)^{b}
e^{\frac{\gamma}{2} X( x) } d x)^p ].
\end{equation}
This quantity is the moment $p$ of the total mass of our GMC measure with two ``insertion points" in $0$ and $1$ of weight $a$ and $b$. The theory of Gaussian multiplicative chaos tells us that these moments are non-trivial, i.e. different from 0 and $ + \infty $, if and only if:
\begin{equation}\label{bounds}
a > - \frac{\gamma^2}{4} - 1, \quad  b > - \frac{\gamma^2}{4} - 1, \quad p < \frac{4}{\gamma^2} \wedge (1+\frac{4}{\gamma^2}( 1 + a)) \wedge (1+\frac{4}{\gamma^2}( 1 + b)).
\end{equation}
The first two conditions are required for the GMC measure to integrate the fractional powers $x^a$ and $(1-x)^b$. Notice that this condition is weaker than the one we would get with the Lebesgue measure, $a>-1$ and $b >-1$.\footnote{Proving Theorem \ref{main_result} for $-1 - \frac{\gamma^2}{4} <a\le -1$ will require a lot of technical work as precise estimates on GMC measures are required to show that Proposition \ref{BPZ} holds in this case.} We then have a bound on the moment $p$, the first part $p < \frac{4}{\gamma^2}$ is the standard condition for the existence of a moment of GMC without insertions. The additional condition on $p$, $p <  (1+\frac{4}{\gamma^2}( 1 + a)) \wedge (1+\frac{4}{\gamma^2}( 1 + b))$, comes from the presence of the insertions. A proof of the bounds \eqref{bounds} can be found in \cite{Houches, Disk}.

Now the goal of our paper is simply to prove the following exact formula for $ M(\gamma, p, a,b ) $:

\begin{theorem}\label{main_result}
For $\gamma \in (0,2)$ and for $p$, $a$, $b$ satisfying \eqref{bounds}\footnote{The result also holds for all complex $p$ such that $\text{Re}(p)$ satisfies the bounds \eqref{bounds}.}, $M(\gamma, p, a,b )$ is given by,
\begin{footnotesize}
\begin{equation*}
  \frac{(2 \pi)^p  \Gamma_{\frac{\gamma}{2}}( \frac{2}{\gamma}(a+1 ) -(p-1) \frac{\gamma}{2}  )
\Gamma_{\frac{\gamma}{2}}( \frac{2}{\gamma}(b+1 ) -(p-1) \frac{\gamma}{2}  ) \Gamma_{\frac{\gamma}{2}}(
\frac{2}{\gamma}(a+b+2 ) -(p-2) \frac{\gamma}{2}  ) \Gamma_{\frac{\gamma}{2}} ( \frac{2}{\gamma} -p \frac{\gamma}{2} ) }{(\frac{\gamma}{2})^{p\frac{\gamma^2}{4} } \Gamma( 1  -  \frac{\gamma^2}{4} ) ^p \Gamma_{\frac{\gamma}{2}}(\frac{2}{\gamma}) \Gamma_{\frac{\gamma}{2}}( \frac{2}{\gamma}(a+1 ) +
\frac{\gamma}{2}  ) \Gamma_{\frac{\gamma}{2}}( \frac{2}{\gamma}(b+1 ) + \frac{\gamma}{2}  )
\Gamma_{\frac{\gamma}{2}}( \frac{2}{\gamma}(a+b+2 ) -(2p-2) \frac{\gamma}{2}  )  
},
\end{equation*}
\end{footnotesize}
where the function $\Gamma_{\frac{\gamma}{2}}(x)$ is defined for $x >0$ and $ Q = \frac{\gamma}{2} + \frac{2}{\gamma} $ by:
\begin{small}
\begin{equation}\label{double_gamma}
\ln \Gamma_{\frac{\gamma}{2}}(x) = \int_0^{\infty} \frac{dt}{t} \left[ \frac{ e^{-xt} -e^{- \frac{Qt}{2}}   }{(1 - e^{- \frac{\gamma t}{2}})(1 - e^{- \frac{2t}{\gamma}})} - \frac{( \frac{Q}{2} -x)^2 }{2}e^{-t} + \frac{ x -\frac{Q}{2}  }{t} \right]. 
\end{equation}
\end{small}
\end{theorem}
As a corollary by choosing $a=b=0$ we obtain the value of the moments of the GMC measure without insertions:
\begin{corollary} For $\gamma \in (0,2)$ and $p < \frac{4}{\gamma^2}$:
\begin{small}
\begin{equation*}
\mathbb{E}[ (\int_{0}^{1}   e^{\frac{\gamma}{2} X( x) } d x)^p ]  = \frac{(2 \pi)^p (\frac{2}{\gamma})^{p\frac{\gamma^2}{4} } }{\Gamma( 1  -  \frac{\gamma^2}{4} ) ^p  }   \frac{\Gamma_{\frac{\gamma}{2}}( \frac{2}{\gamma} -(p-1) \frac{\gamma}{2}  )^2
 \Gamma_{\frac{\gamma}{2}}(
\frac{4}{\gamma} -(p-2) \frac{\gamma}{2}  ) \Gamma_{\frac{\gamma}{2}} ( \frac{2}{\gamma} -p\frac{\gamma}{2} ) }{\Gamma_{\frac{\gamma}{2}}(\frac{2}{\gamma}) \Gamma_{\frac{\gamma}{2}}( \frac{2}{\gamma} +
\frac{\gamma}{2}  )^2 
\Gamma_{\frac{\gamma}{2}}( \frac{4}{\gamma} -(2p-2) \frac{\gamma}{2}  )  
}. 
\end{equation*}
\end{small}
\end{corollary}

Thanks to the computations performed by Ostrovsky \cite{Ostro3}, we can also state our main result in the following equivalent way:

\begin{corollary}\label{law_ostro} The following equality in law holds,
\begin{equation}
\int_{0}^{1}  x^{a}(1-x)^{b} e^{\frac{\gamma}{2} X( x) } d x = 2 \pi 2^{-(3(1 + \frac{\gamma^2}{4})+ 2( a + b))   }  L Y_{\gamma} X_1 X_2 X_3 ,
\end{equation}
where $L, Y_{\gamma}, X_1, X_2, X_3$ are five independent random variables in $\mathbb{R_+}$ with the following laws:
\begin{align*}
L &= \exp (\mathcal{N}(0, \gamma^2 \ln 2 )) \\
Y_{\gamma} &= \frac{1}{\Gamma(1 - \frac{\gamma^2}{4})} \mathcal{E}(1)^{-\frac{\gamma^2}{4}} \\
X_1 &= \beta_{2,2}^{-1}(1, \frac{4}{\gamma^2} ;1 +  \frac{4}{\gamma^2}(1 + a),  \frac{2(b-a)}{\gamma^2}, \frac{2(b-a)}{\gamma^2}), \\
X_2 &= \beta_{2,2}^{-1}( 1,\frac{4}{\gamma^2} ;1 +  \frac{2}{\gamma^2}(2 + a + b),  \frac{1}{2},  \frac{2}{\gamma^2}  ), \\
X_3 &= \beta_{2,2}^{-1}( 1,\frac{4}{\gamma^2} ; 1 + \frac{4}{\gamma^2},  \frac{1}{2} + \frac{2}{\gamma^2}(1 + a +b ) ,  \frac{1}{2} + \frac{2}{\gamma^2}(1 + a +b )  ).
\end{align*}
Here $\mathcal{E}(1)$ is an exponential law of parameter $1$ and $\beta_{2,2}$ is a special beta law defined in appendix \ref{sec_special}. It satisfies $\beta_{2,2} \in [0,1]$.
\end{corollary}

The advantage of this formulation is that it is more transparent than the large formula of Theorem \ref{main_result}. The log-normal law $L$ is a global mode coming from the fact that $X$ is not of zero average on $[0,1]$, see the discussion of subsection \ref{sec_small_dev}. The random variable $Y_{\gamma}$ is actually the law of the total mass of the GMC measure defined on the unit circle - see \cite{remy} - and it will play a crucial role in understanding the small deviations of GMC, see again subsection \ref{sec_small_dev}. Lastly the generalized beta laws studied in \cite{Ostro2} have a complicated definition but take values in $[0,1]$ just like the standard beta law.

\subsection{Strategy of the proof}\label{strategy}
We start off with the well known observation that a formula can be given for $ M(\gamma, p, a,b ) $ in the very special case where $p \in \mathbb{N}$, $a>-1$, $b> -1$ and $p$ satisfying \eqref{bounds}. Indeed, in this case the computation reduces to a real integral - the famous Selberg integral - whose value is known, see for instance \cite{Selberg}. This is because for a positive integer moment we can write $p$ integrals and exchange them with the expectation $\mathbb{E}[\cdot]$. More precisely for $a,b>-1$, $p$ satisfying \eqref{bounds} and $p \in \mathbb{N}$ we have, using any suitable regularization procedure:
\begin{align}\label{Selberg}
\mathbb{E}&[ (\int_{0}^{1}  x^{a}(1-x)^{b} e^{\frac{\gamma}{2} X( x) } d x)^p ]  \nonumber \\ 
 &= \lim_{\delta \rightarrow 0} \mathbb{E}[ (\int_0^1 x^a (1-x)^b e^{ \frac{\gamma}{2}X_{\delta}(x) - \frac{\gamma^2}{8} \mathbb{E}[ X_{\delta}(x)^2]  } dx)^p ] \nonumber \\ 
&= \lim_{\delta \rightarrow 0}  \int_{[0,1]^p} \prod_{i=1}^p x_i^a (1-x_i)^b \mathbb{E}[ \prod_{i=1}^p e^{ \frac{\gamma}{2}X_{\delta}(x_i) - \frac{\gamma^2}{8} \mathbb{E}[ X_{\delta}(x_i)^2]  }  ] dx_1 \dots dx_p   \nonumber \\ 
&=  \int_{[0,1]^p} \prod_{i=1}^p x_i^a (1-x_i)^b e^{ \frac{\gamma^2}{4} \sum_{i<j} \mathbb{E}[ X(x_i) X(x_j) ] } dx_1 \dots dx_p \nonumber \\ 
&= \int_{[0,1]^p} \prod_{i=1}^p x_i^a (1-x_i)^b  \prod_{i<j} \frac{1}{\vert x_i - x_j \vert^{\frac{\gamma^2}{2}}}  dx_1 \dots dx_p   \nonumber \\
&= \prod_{j=1}^p \frac{\Gamma(1 + a -(j-1)\frac{\gamma^2}{4} )
\Gamma(1 + b -(j-1)\frac{\gamma^2}{4} ) \Gamma( 1 - j \frac{\gamma^2}{4}  ) 
}{\Gamma(2 + a + b - (p + j -2)\frac{\gamma^2}{4}  ) \Gamma(1 - \frac{\gamma^2}{4}
)}.
\end{align}
The last line is precisely given by the Selberg integral. It is then natural to look for an analytic continuation of this expression from positive integer $p$ to any real $p$ satisfying \eqref{bounds}. Notice that giving the analytic continuation of a such a quantity is a highly non-trivial problem as $p$ appears both in the argument of the Gamma functions as well as in the number of terms in the product. To find the right candidate for the analytic continuation we start by writing down the following relations that we will refer to as the shift equations. They are deduced by simple algebra from \eqref{Selberg} again for $p \in \mathbb{N}$ and under the bounds \eqref{bounds},
\begin{small}
\begin{align}\label{shift_equa1}
\frac{M(\gamma, p, a + \frac{\gamma^2}{4} ,b )}{M(\gamma, p, a,b )}  &=  
\frac{\Gamma(1 + a + \frac{\gamma^2}{4} ) \Gamma(2 + a + b - (2p -2
)\frac{\gamma^2}{4}  )  }{\Gamma(1 + a -(p-1) \frac{\gamma^2}{4} ) \Gamma(2 + a + b
- (p -2 )\frac{\gamma^2}{4}  ) },  \\ \label{shift_equa2}
\frac{M(\gamma, p, a + 1 ,b )}{M(\gamma, p, a,b )}  &=  
\frac{\Gamma(\frac{4}{\gamma^2}(1 +a) +1 ) \Gamma( \frac{4}{\gamma^2
} (2 + a +b) - (2p-2)) }{\Gamma(\frac{4}{\gamma^2}(1 +a) -(p-1 )) \Gamma(
\frac{4}{\gamma^2
} (2 + a +b) - (p-2)) }, 
\end{align}
\end{small}
and for $p \in \mathbb{N^*}$ under the bounds \eqref{bounds},
\begin{footnotesize}
\begin{align}\label{shift_equa3}
&\frac{M(\gamma, p, a,b )}{M(\gamma, p-1, a,b )}  \\ \nonumber
  &= \frac{\Gamma(1 + a -
(p-1)\frac{\gamma^2}{4} ) \Gamma(1 + b - (p-1)\frac{\gamma^2}{4} ) \Gamma( 1 - p
\frac{\gamma^2}{4}  ) \Gamma(2 + a + b - (p -2 )\frac{\gamma^2}{4}  )  }{\Gamma(2 +
a + b - (2p -3 )\frac{\gamma^2}{4}  ) \Gamma(2 + a + b - (2p-2  )\frac{\gamma^2}{4}  )
\Gamma(1 - \frac{\gamma^2}{4} )}.   
\end{align}
\end{footnotesize}

Of course similar shift equations hold for $b$ but as there is a symmetry $M(\gamma,p,a,b)=M(\gamma,p,b,a)$ we will write everything only for $a$. The reason why the function $\Gamma_{\frac{\gamma}{2}}(x)$ introduced in Theorem \ref{main_result} appears is that it verifies the following two relations, for $ \gamma \in (0,2)$ and $x >0$,
\begin{align}
\frac{\Gamma_{\frac{\gamma}{2}}(x)}{\Gamma_{\frac{\gamma}{2}}(x + \frac{\gamma}{2}) }&= \frac{1}{\sqrt{2 \pi}}
\Gamma(\frac{\gamma x}{2}) ( \frac{\gamma}{2} )^{ -\frac{\gamma x}{2} + \frac{1}{2}
}, \\
\frac{\Gamma_{\frac{\gamma}{2}}(x)}{\Gamma_{\frac{\gamma}{2}}(x + \frac{2}{\gamma}) }&= \frac{1}{\sqrt{2 \pi}} \Gamma(\frac{2
x}{\gamma}) ( \frac{\gamma}{2} )^{ \frac{2 x}{\gamma} - \frac{1}{2} }.
\end{align}
See appendix \ref{sec_special} for more details on $\Gamma_{\frac{\gamma}{2}}(x)$. Therefore we can use $\Gamma_{\frac{\gamma}{2}}(x)$ to construct a candidate function that will verify all the shift equations \eqref{shift_equa1}, \eqref{shift_equa2}, \eqref{shift_equa3} not only for $p \in  \mathbb{N}$ but for any real $p$ satisfying the bounds \eqref{bounds}. More precisely for any function $C(p)$ of $p$ (and $\gamma$) the following quantity,
\begin{small}
\begin{equation}\label{expression_a_b}
  C(p) \frac{\Gamma_{\frac{\gamma}{2}}( \frac{2}{\gamma}(a+1 ) -(p-1) \frac{\gamma}{2}  )
\Gamma_{\frac{\gamma}{2}}( \frac{2}{\gamma}(b+1 ) -(p-1) \frac{\gamma}{2}  ) \Gamma_{\frac{\gamma}{2}}(
\frac{2}{\gamma}(a+b+2 ) -(p-2) \frac{\gamma}{2}  )  }{\Gamma_{\frac{\gamma}{2}}( \frac{2}{\gamma}(a+1 ) +
\frac{\gamma}{2}  ) \Gamma_{\frac{\gamma}{2}}( \frac{2}{\gamma}(b+1 ) + \frac{\gamma}{2}  )
\Gamma_{\frac{\gamma}{2}}( \frac{2}{\gamma}(a+b+2 ) -(2p-2) \frac{\gamma}{2}  )  
}, 
\end{equation}
\end{small}
is a solution to the shift equations \eqref{shift_equa1}, \eqref{shift_equa2}. Notice that for $\frac{\gamma^2}{4} \notin \mathbb{Q}$ these two shift equations completely determine the dependence on $a$ (and on $b$ by symmetry) of $M(\gamma, p,a,b)$. Then by a standard continuity argument in $\gamma$ we will be able to extend the expression \eqref{expression_a_b} to all $\gamma \in (0,2)$. Next the equation \eqref{shift_equa3} translates into a constraint on the unknown function $C(p)$:
\begin{equation}\label{relation_c1}
\frac{C(p)}{C(p-1)}  = \sqrt{2 \pi}  (\frac{\gamma}{2})^{(p-1) \frac{\gamma^2}{4}-\frac{1}{2} }  \frac{\Gamma(1 -p \frac{\gamma^2}{4} )   }{ \Gamma( 1  -  \frac{\gamma^2}{4} )  }.
\end{equation}
We see that \eqref{relation_c1} is not enough to fully determine the function $C(p)$. An additional shift equation that is a priori not predicted by the Selberg integral \eqref{Selberg} is required. We will indeed prove that we have,
\begin{equation}\label{relation_c2}
\frac{C(p)}{C(p - \frac{4}{\gamma^2})} = f(\gamma) ( \frac{\gamma}{2} )^{-p} \Gamma(\frac{4}{\gamma^2} -p),
\end{equation}
where $f(\gamma)$ is an unknown positive function of $\gamma$. Now combining \eqref{relation_c1} and \eqref{relation_c2} completely determines the function $C(p)$ again up to an unknown constant $c_{\gamma}$ of $\gamma$:
\begin{equation}\label{shift_c1}
C(p) =  c_{\gamma} \frac{(2 \pi)^p}{\Gamma( 1  -  \frac{\gamma^2}{4} ) ^p}  (\frac{2}{\gamma})^{p\frac{\gamma^2}{4}}  \Gamma_{\frac{\gamma}{2}} ( \frac{2}{\gamma} -p \frac{\gamma}{2} ).
\end{equation}
This last constant $c_{\gamma}$ is evaluated by choosing $p=0$ and thus we arrive at the function of Theorem \ref{main_result} giving the expression of $M(\gamma, p ,a ,b)$.

Now the major difficulty that must be overcome is to find a way to prove all the shift equations \eqref{shift_equa1}, \eqref{shift_equa2}, \eqref{shift_equa3} as well as the additional equation \eqref{relation_c2} for all real values of $p,a,b$ satisfying \eqref{bounds} and not just for positive integer $p$. To achieve this the key ingredient of our proof is to introduce the following two auxiliary functions for $t\le 0$,
\begin{equation}
U(t) :=   \mathbb{E}[ (\int_{0}^{1} (x-t)^{\frac{\gamma^2}{4}}  x^{a}(1-x)^{b} e^{\frac{\gamma}{2} X( x) } d x)^p ],
\end{equation}
and 
\begin{equation}
\tilde{U}(t) :=  \mathbb{E}[ (\int_{0}^{1} (x-t) \, x^{a}(1-x)^{b} \,e^{\frac{\gamma}{2} X( x) } d x)^p ],
\end{equation}
and to show using probabilistic techniques that the following holds: 
\begin{proposition}\label{BPZ}
For $\gamma \in (0,2) $, $a,b,p$ satisfying \eqref{bounds} and $t <0$, $U(t)$ is solution of the hypergeometric equation:
\begin{equation}\label{ODE U}
t(1-t)U''(t)+(C-(A+B+1)t)U'(t)-ABU(t)=0. 
\end{equation}
The parameters $A, B, C$ are given by:
\begin{equation}
A= -\frac{p\gamma^2}{4}, B=-(a+b+1)-(2-p)\frac{\gamma^2}{4}, C = -a-\frac{\gamma^2}{4}.
\end{equation}
Similarly $\tilde{U}(t)$ is solution of the hypergeometric equation but with parameters $\tilde{A}, \tilde{B}, \tilde{C}$ given by:
\begin{equation}
\tilde{A}= -p, \tilde{B}=-\frac{4}{\gamma^2}(a+b+2)+p-1, \tilde{C} = -\frac{4}{\gamma^2}(a+1).
\end{equation}
\end{proposition}
Let us make a few comments on the meaning of $U(t)$ and $\tilde{U}(t)$. These auxiliary functions are very similar to the correlation functions of LCFT with a degenerate field insertion - see \cite{DOZZ1,DOZZ2} for the case of the sphere and \cite{remy} for the unit disk - which also obey differential equations known as the BPZ equations. What is mysterious in our present case is that it is not clear whether there exists an actual CFT where $U(t)$ and $\tilde{U}(t)$ correspond to correlations with degenerate insertions which would explain why the differential equations of Proposition \ref{BPZ} hold. Furthermore if we replace the real $t$ by a complex variable $t\in \mathbb{C}\backslash [0,\infty]$, it is not hard to see that $U(t)$ is a holomorphic function and Proposition \ref{BPZ} will hold if we replace the ordinary derivative by a complex derivative $\partial_t$. In the conformal bootstrap approach of CFT initiated by Belavin-Polyakov-Zamolodchikov in \cite{BPZ}, a correlation function with a degenerate insertion can be decomposed into combinations of the structure constants and of the conformal blocks. A conformal block is a locally holomorphic function and it is always accompanied by its complex conjugate in the decomposition. What is mysterious with $U(t)$ and $\tilde{U}(t)$ is that we only see the holomorphic part. At this stage we have no CFT-based explanation for this observation although a possible path could be to look at boundary LCFT with multiple boundary cosmological constants, see for instance \cite{nakayama}. On the other hand let us mention that again in the very special case where $ p \in \mathbb{N}$, $U(t)$ and $\tilde{U}(t)$ reduce to Selberg-type integrals and the equations of Proposition \ref{BPZ} were known in this case, see \cite{kaneko}.  

Proposition \ref{BPZ} will be established in section \ref{section DE} by performing direct computations on $U(t)$ and $\tilde{U}(t)$. We then write the solutions of the hypergeometric equations in two different bases. One solution corresponds to a power series expansion in $\vert t \vert$ and the other to an expansion in $ \vert t \vert^{-1}$. The change of basis formula \eqref{hpy1} written in appendix \ref{sec_special} given by the theory of hypergeometric equations then provides non-trivial relations which are precisely the shift equations that we wish to prove. This is performed in detail in section \ref{sec_shift} where Proposition \ref{prop_a} completely determines the dependence in $a$ and $b$ of $M(\gamma, p ,a ,b)$ and Proposition \ref{prop_p} establishes \eqref{shift_c1}. Thus we have proved Theorem \ref{main_result}.

\subsection{Tail expansion for GMC and the reflection coefficients}\label{section tail}

Before moving into the proof of our main result, we provide in this subsection and in the following some applications of Theorem \ref{main_result}. The first application we will consider deals with tail expansions for GMC measures, in other words the probability for a GMC measure to be large. We choose to include here a very general discussion about these tail expansions of GMC with an arbitrary insertion both in one and in two dimensions. For each tail expansion result there will appear a universal coefficient known as the reflection coefficient.

The first case that was studied is the tail expansion of a GMC in dimension two and a precise asymptotic was given in \cite{DOZZ2} in terms of the reflection coefficient $\overline{R}_2(\alpha)$,\footnote{In \cite{DOZZ2} or \cite{ReviewDOZZ} this coefficient is actually called $\overline{R}(\alpha)$ but for the needs of our discussion we introduce the $2$ to indicate the dimension. Furthermore the bar stands for the fact that it is the unit volume coefficient.} see Proposition \ref{tail_KRV} below.\footnote{$\overline{R}_2(\alpha)$ is the bulk reflection coefficient in dimension two, a boundary reflection coefficient $\overline{R}^{\partial}_2(\alpha)$ also exists but its value remains unknown, see the figure below.} Let us mention that it was recently discovered in \cite{ReviewDOZZ} that $\overline{R}_2(\alpha)$ corresponds to the partition function of the $\alpha$-quantum sphere introduced by Duplantier-Miller-Sheffield in \cite{Mating}. Now our exact formula on the unit interval will allow us to write a similar tail expansion for GMC in dimension one. Following \cite{Mating} we use the standard radial decomposition of the covariance \eqref{covariance} of $X$ around the point $0$, i.e. we write for $ s \geq 0 $,
\begin{equation}\label{radial}
X( e^{-s/2}) = B_s + Y( e^{-s/2}),
\end{equation}
where $B_s$ is a standard Brownian motion and $Y$ is an independent Gaussian process that can be defined on the whole plane with covariance given for $x, y \in \mathbb{C}$ by:
\begin{equation}
\mathbb{E}[ Y(x) Y(y)] = 2 \ln \frac {|x| \vee |y|}{\vert x -y \vert}.
\end{equation}

Motivated by the Williams decomposition of Theorem \ref{theo Williams}, we introduce for $\lambda>0$ the process that will be used in the definitions below,
\begin{equation}\label{mathcal B}
\mathcal{B}^{\lambda}_s := 
\begin{cases}
\hat{B}_s-\lambda s  \quad s\ge 0\\
\bar{B}_{-s}+\lambda s \quad s<0,
\end{cases}
\end{equation}
where $(\hat{B}_s-\lambda s)_{s\ge 0}$ and $(\bar{B}_{s}-\lambda s)_{s\ge 0}$ are two independent Brownian motions with negative drift conditioned to stay negative.
We can now give the definitions of the two coefficients in dimension one $ \overline{R}^{\partial}_{1}(\alpha)$ and $\overline{R}_{1}(\alpha)$ along with the associated GMC measures with insertion $I^{\partial}_{1, \eta}( \alpha )$ and $I_{1, \eta}( \alpha )$ whose tail behavior will be governed by the corresponding coefficient:
\begin{small}
\begin{align*}
 \overline{R}^{\partial}_{1}(\alpha) &:=\E [(\frac{1}{2} \int_{-\infty}^{\infty} e^{ \frac{\gamma}{2} \mathcal{B}_s^{\frac{Q-\alpha}{2}} } e^{\frac{\gamma}{2} Y(e^{-s/2})  } ds)^{\frac{2}{\gamma}(Q - \alpha)}  ], \\
\overline{R}_{1}(\alpha) &:= \E [(\frac{1}{2} \int_{-\infty}^{\infty} e^{ \frac{\gamma}{2} \mathcal{B}_s^{\frac{Q-\alpha}{2}} } (e^{\frac{\gamma}{2} Y(e^{-s/2})  } + e^{\frac{\gamma}{2} Y(-e^{-s/2})  } ) ds)^{\frac{2}{\gamma}(Q -\alpha)}],\\
 I^{\partial}_{1, \eta}( \alpha ) &:= \int_0^{\eta} x^{-\frac{\gamma \alpha}{2}} e^{\frac{\gamma}{2} X(x)} dx,   \\
 I_{1, \eta}( \alpha ) &:= \int_{v -\eta}^{v + \eta} \vert x -v \vert^{-\frac{\gamma \alpha}{2}} e^{\frac{\gamma}{2} X(x)} dx.
\end{align*}
\end{small}
Let us make some comments on these definitions. Here $\alpha \in (\frac{\gamma}{2}, Q)$, $Q = \frac{\gamma}{2} + \frac{2}{\gamma}$, and $\eta$ is an arbitrary positive real number chosen small enough. To match the conventions of the study of LCFT we have written the fractional power $x^{-\frac{\gamma \alpha}{2}}$, so in these notations we have $a = - \frac{\gamma \alpha}{2}$. Notice that the difference between $I^{\partial}_{1, \eta}( \alpha )$ and $I_{1, \eta}( \alpha ) $ lies in the position of the insertion. For $I^{\partial}_{1, \eta}( \alpha )$ the insertion is placed in $0$ (by symmetry we could have placed it in $1$). Our Theorem \ref{main_result} will give us the value of the associated coefficient $\overline{R}^{\partial}_{1}(\alpha)$. The other case corresponds to placing the insertion at a point $v$ inside the interval, $v \in (0,1)$, and gives the quantity $I_{1, \eta}( \alpha )$. The computation of the associated $\overline{R}_{1}(\alpha)$ will be done in a future work. We now claim:
\begin{proposition}\label{proposition reflection}
For $\alpha \in ( \frac{\gamma}{2}, Q) $ we have the following tail expansion for $I_{1, \eta}^{\partial}(\alpha)$ as $u \rightarrow \infty$ and for some $\nu >0$,
\begin{equation}\label{tail_result1}
\mathbb{P} ( I_{1, \eta}^{\partial}(\alpha) >u )  = \frac{ \overline{R}_1^{\partial}(\alpha)}{u^{ \frac{2}{\gamma}(Q - \alpha) }} + O ( \frac{1}{u^{ \frac{2}{\gamma}(Q - \alpha)+ \nu }  }   ), 
\end{equation}
where the value of $  \overline{R}_1^{\partial}(\alpha)$ is given by:
\begin{equation}\label{tail_result2}
 \overline{R}_1^{\partial}(\alpha) =  \frac{ (2 \pi)^{ \frac{2}{\gamma}(Q -\alpha ) -\frac{1}{2}} (\frac{2}{\gamma})^{ \frac{\gamma}{2}(Q -\alpha ) -\frac{1}{2} }  }{(Q-\alpha) \Gamma(1 -\frac{\gamma^2}{4}  )^{ \frac{2}{\gamma}(Q -\alpha ) } } \frac{ \Gamma_{\frac{\gamma}{2}}(\alpha - \frac{\gamma}{2}  )}{\Gamma_{\frac{\gamma}{2}}(Q- \alpha )}.
\end{equation}
\end{proposition}
The proof of this proposition is done in appendix \ref{sec_reflection}. Notice that we impose the condition $ \alpha \in (\frac{\gamma}{2}, Q)$. This is crucial for the tail behavior of $I^{\partial}_{1, \eta}( \alpha )$ (or similarly for $I_{1, \eta}( \alpha )$) to be dominated by the insertion and this is precisely why the asymptotic expansion is independent of the choice of $\eta$. It also explains why the radial decomposition \eqref{radial} is natural as it is well suited to study $X$ around a particular point. If one is interested in the case where $\alpha < \frac{\gamma}{2}$ (or simply $\alpha =0$), a different argument known as the localization trick is required to obtain the tail expansion, see \cite{tail} for more details.
 For the sake of completeness of our discussion we also recall the tail expansion in dimension two that was obtained in \cite{DOZZ2}. The normalizations in this case are slightly different as we do not include a factor $2$ in the covariance. We work with a Gaussian process $\tilde{X}$ defined on the unit disk $\mathbb{D}$ with covariance $\ln \frac{1}{\vert x -y\vert}$. Instead of $Y$ we use $\tilde{Y}$ with covariance:
\begin{equation}
\mathbb{E}[ \tilde{Y}(x) \tilde{Y}(y)] = \ln \frac {|x| \vee |y|}{\vert x -y \vert}.
\end{equation}
For an insertion placed in $z$, $\vert z \vert <1$ we now define,
\begin{align*}
\overline{R}_{2}(\alpha) &:= \E[( \int_{-\infty}^{\infty} e^{ \gamma \mathcal{B}_s^{Q-\alpha} } \int_0^{2 \pi} e^{\gamma \tilde{Y}(e^{-s} e^{i \theta} )  } ds)^{\frac{2}{\gamma}(Q-\alpha)}], \\
  I_{2,\eta}( \alpha ) &:= \int_{B(z,\eta)} \vert x - z  \vert^{-\gamma \alpha} e^{\gamma \tilde{X}(x)} d^2x,
\end{align*}
and we state the result obtained in \cite{DOZZ2}:
\begin{proposition}\label{tail_KRV} (Kupiainen-Rhodes-Vargas \cite{DOZZ2})
For $\alpha \in ( \frac{\gamma}{2}, Q) $ we have the following tail expansion for $I_{2,\eta}(\alpha)$ as $u \rightarrow \infty$ and for some $\nu >0$,
\begin{equation}
\mathbb{P} ( I_{2,\eta}(\alpha) >u )  = \frac{ \overline{R}_2(\alpha)}{u^{ \frac{2}{\gamma}(Q - \alpha) }} + O ( \frac{1}{u^{ \frac{2}{\gamma}(Q - \alpha)+ \nu }  }   ), 
\end{equation}
where the value of $  \overline{R}_2(\alpha)$ is given by:
\begin{equation}
 \overline{R}_2(\alpha) = - \frac{\gamma}{2(Q-\alpha)} \frac{(\pi \Gamma( \frac{\gamma^2}{4} ))^{\frac{2}{\gamma}(Q -\alpha ) } }{ \Gamma(1 - \frac{\gamma^2}{4} )^{\frac{2}{\gamma}(Q -\alpha ) }} \frac{\Gamma(-\frac{\gamma}{2}(Q -\alpha ) ) }{ \Gamma(\frac{\gamma}{2}(Q -\alpha )) \Gamma(\frac{2}{\gamma}(Q -\alpha ))   }.
\end{equation}
\end{proposition}
A similar proposition is also expected for $  \overline{R}_2^{\partial}(\alpha)$, the boundary reflection coefficient in dimension two, whose expression and computation are left for a future paper. One notices that $ \overline{R}_1^{\partial}(\alpha)$ has a more convoluted expression than $ \overline{R}_2(\alpha) $ as the special function $\Gamma_{\frac{\gamma}{2}}$ appears in its expression. Such expressions have already appeared in the study of Liouville theory for instance in \cite{Ponsot} where a general formula for the reflection amplitude is given. We now summarize the four different cases that we have discussed in the following figure. For each coefficient the number $1$ or $2$ stands for the dimension and the partial $\partial$ symbol stands for the boundary cases, no $\partial$ corresponds to the bulk cases. 

\begin{figure}[!htp]
\centering
\includegraphics[width=0.7\linewidth]{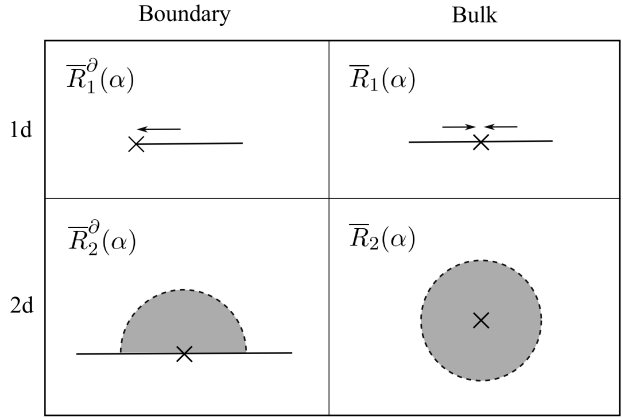}
\caption{Four types of reflection coefficients}
\end{figure}

\subsection{Small deviations for GMC}\label{sec_small_dev}

We now turn to the problem of determining the universal behavior of the probability for a GMC to be small. Both the exact formulas of Theorem \ref{main_result} and the one proven on the unit circle in \cite{remy} will provide crucial insight. For this subsection only we will use the following shorthand notation:
\begin{equation}
I_{\gamma,a,b} := \int_{0}^{1}  x^{a}(1-x)^{b} e^{\frac{\gamma}{2} X( x) } d x.
\end{equation}

In the following we will rely extensively on the decomposition $$ I_{\gamma,a,b} = \tilde{c} L Y_{\gamma} X_1 X_2 X_3$$ coming from Corollary \ref{law_ostro} with $\tilde{c}$ being a positive constant. First $L$ is a log-normal law, so one has $ \mathbb{P}(L \leq \epsilon) \leq c_1 \exp( -c_2 (\ln \epsilon)^2 ) $ for some $c_1, c_2 > 0$. On the other hand the probability for $Y_{\gamma}$ to be small is much smaller since $\mathbb{P}(Y_{\gamma} \leq \epsilon ) \leq \exp( -c \epsilon^{-\frac{4}{\gamma^2}}) $ for some $c>0$. From the above and since $X_1, X_2, X_3 \geq 1$ the probability to be small for $I_{\gamma,a,b}$ will be of log-normal type. By comparison in the case of the total mass of the GMC on the unit circle it was shown in \cite{remy} that it is distributed according to $Y_{\gamma}$ and so its probability to be small is of order $\exp( -c \epsilon^{-\frac{4}{\gamma^2}}) $.

Thus it appears that GMC on the unit interval and the unit circle have completely different small deviations. However this difference comes from the fact that the log-correlated field on the circle is of average zero while in the case of the interval there is a non-zero global mode producing the log-normal variable $L$.  Therefore on the interval if one subtracts the average of $X$ with respect to the correct measure (see below) one can remove the log-normal law $L$ appearing in the decomposition of Corollary \ref{law_ostro}. The probability for the resulting GMC to be small will then be bounded by $\exp(- c \epsilon^{ - \frac{4}{\gamma^2}}) $ for some $c>0$ just like for the case of the circle. We expect this to be the correct universal behavior.

Let us make the above more precise. We start by writing down the decomposition of the covariance of our field in terms of the Chebyshev polynomials. For all $x,y \in [0,1]$ with $x \neq y$ we have:
\begin{equation}
- 2 \ln | x -y | = 4 \ln 2 + \sum_{n=1}^{+ \infty} \frac{4}{n} T_n(2x -1) T_n(2y -1).
\end{equation}
We recall that the Chebyshev polynomial of order $n$ is the unique polynomial verifying $T_n(\cos \theta) = \cos (n \theta) $. This basis of polynomials is also orthogonal with respect to dot product given by the integration against $\frac{1}{\sqrt{1 -x^2}} dx $, i.e.
\begin{equation}
\int_{-1}^1 T_n(x) T_m(x) \frac{1}{\sqrt{1 -x^2}} dx =  \left\{ \begin{array}{lcl} 0 & \text{for}  & n \neq m \\ \pi & \text{for}  &  n=m=0 \\ \frac{\pi}{2}  & \text{for}  &  n=m \neq 0 \end{array} \right.   \
\end{equation} 
From the above our field $X(x)$ can be constructed by the series:
\begin{equation}\label{tche}
X(x) = 2 \sqrt{\ln 2} \alpha_0 + \sum_{n=1}^{+ \infty} \frac{2 \alpha_n}{\sqrt{n}} T_n(2x-1).
\end{equation}
Here $(\alpha_n)_{n \in \mathbb{N}}$ is a sequence of i.i.d. standard Gaussians. This of course only makes sense if one integrates both sides against a test function. We now introduce:
$$ \overline{X} := \frac{2}{\pi}\int_{0}^1  \frac{1}{\sqrt{1 -(2x-1)^2}} X(x) dx = 2 \sqrt{\ln 2} \alpha_0 \quad  \text{and} \quad X_{\perp}(x) := X(x) - \overline{X}.$$
We easily check that  $e^{\frac{\gamma}{2} \overline{X}} \overset{\text{law}}{=} \exp({\mathcal{N}(\gamma^2 \ln 2)})  $. The probability to be small for the GMC associated to $  X_{\perp}(x) $ is now given by,
\begin{equation}\label{small_dev}
\mathbb{P}( \int_0^{1} e^{\frac{\gamma}{2} X_{\perp}(x) } d x \leq \epsilon ) \le  \exp(-c \epsilon^{-\frac{4}{\gamma^2}}).
\end{equation}
This result can be easily obtained from Corollary \ref{law_ostro} by noticing that since we removed $L = \exp({\mathcal{N}(\gamma^2 \ln 2)})$ the probability to be small is now governed by $Y_{\gamma}$ which gives the bound written above. The argument we have just described is expected to work for any GMC in any dimension, a result of this nature can be found in \cite{lacoin}.

There is also a direct application of these observations to determining the law of the random variable $I_{\gamma,a,b}$. This is linked to how the strategy of the proof of the present paper differs from the one used in \cite{remy} to prove the Fyodorov-Bouchaud formula. In subsection \ref{subsection shift p} we first use the differential equation \eqref{ODE U} on $U(t)$ to obtain a relation between $M(\gamma, p, a,b ) $ and  $M(\gamma, p-1, a,b ) $. Thus from this relation and knowing that $M(\gamma, 0, a,b ) =1 $ one can compute recursively all the negative moments of the random variable $I_{\gamma,a,b}$. As it was emphasized in many papers (see the review \cite{Ostro_review} by Ostrovsky and references therein), the negative moments of $I_{\gamma,a,b}$ do not determine its law as the growth of the negative moments is too fast. This is why we must derive a second relation between $M(\gamma, p, a,b ) $ and  $M(\gamma, p-\frac{4}{\gamma^2}, a,b ) $ which gives enough information to complete the proof. By contrast in the case of the total mass of the GMC on the unit circle the negative moments do capture uniquely the probability distribution and so the proof of the Fyodorov-Bouchaud formula given in \cite{remy} only requires one shift equation (in a similar fashion one obtains a relation between the moment $p$ and the moment $p-1$ of the total mass of the GMC).

But the negative moments of $I_{\gamma,a,b}$ do not determine its law only because of the log-normal law $L$ in the decomposition of Corollary \ref{law_ostro}. By using Corollary \ref{law_ostro} and by independence of $ X_{\perp}(x) $ and  $ \overline{X}$ one can factor out  $ e^{\frac{\gamma}{2} \overline{X}} \overset{\text{law}}{=} L$ and the computation of the negative moments is now sufficient to uniquely determine the distribution. Thus the negative moments of a GMC measure always determine its law if one removes the global Gaussian coming from the average of the field with respect to an appropriate measure. From this observation the relation between $M(\gamma, p, a,b ) $ and  $M(\gamma, p-\frac{4}{\gamma^2}, a,b ) $ could be omitted in the proof of Theorem \ref{main_result}. Nonetheless if one only computes the negative moments it is not clear that the analytic continuation given by the $\Gamma_{\gamma}$ functions does correspond to the fractional moments of a random variable, this fact has been checked by Ostrovsky in \cite{Ostro1}. Thus in order to keep the proof of our theorem self-contained we choose to keep both shift equations.

\subsection{Other applications}\label{sec_app}
Similarly as in \cite{remy} we will write the applications of our Theorem \ref{main_result} to the behavior of the maximum of $X$ and to random matrix theory. We refer to \cite{remy} for more detailed explanations and for additional references on these problems.

Characterizing the behavior of the maximum of $X$ requires to compute the law of the total mass of the derivative martingale,
\begin{align*}
M' &= -\frac{1}{2} \int_0^1 X(x) e^{X(x)} dx\\
 &:= -\frac{1}{2}  \lim_{ \delta \rightarrow 0 } \int_0^1 ( X_{\delta}(x) - \mathbb{E}[X_{\delta}(x)^2])e^{X_{\delta}(x) - \frac{1}{2}  \mathbb{E}[X_{\delta}(x)^2]  }dx,
\end{align*}
which following \cite{APS} can be characterized by the convergence in law:
\begin{equation}
2 M' = \lim_{ \gamma \rightarrow 2 } \frac{1}{2 - \gamma} \int_0^1 e^{\frac{\gamma}{2} X(x) } dx.
\end{equation}
Therefore from our Theorem \ref{main_result} we can easily compute the moments of this quantity,
\begin{align*}
\mathbb{E}[(2 M')^p] &=  (2\pi)^{p} \frac{ \Gamma_1(1-p) \Gamma_1(2 -p)^2 \Gamma_1(4 - p) }{ \Gamma_1(2)^2 \Gamma_1(4 -2p) }\\
 &= \frac{  G(4 -2p) }{ G(1-p) G(2 -p)^2 G(4 - p) },
\end{align*}
where $G(x)$ is the so-called Barnes G function, see appendix \ref{sec_special} for more details. Just like in Corollary \ref{law_ostro} an explicit description of the resulting law has been found in \cite{Ostro2},
\begin{equation}
2 M' \overset{law}{=} \frac{\pi}{32} \tilde{L} \tilde{X}_1 \tilde{X}_2 \tilde{X}_3,
\end{equation}
where $ \tilde{L}, \tilde{X}_1, \tilde{X}_2, \tilde{X}_3$ are four independent random variables on $\mathbb{R}_{+}$ with the following laws:
\begin{align*}
\tilde{L} &= \exp( \mathcal{N}(0, 4 \ln 2)) \\
\tilde{X}_1 &= \frac{1}{y^2} e^{-1/y}, \: y > 0\\
\tilde{X}_2 &= \beta^{-1}_{2,2}(1,1;2,\frac{1}{2}, \frac{1}{2}) \\
\tilde{X}_3 &= \frac{2}{y^3} dy, \: y > 1.
\end{align*}
Then for a suitable regularization $X_{\delta}$ of $X$ the following convergence holds in law:
\begin{align*}
\max_{x \in [0,1]} X_{\delta}(x) - &2 \ln \frac{1}{\delta} + \frac{3}{2}\ln \ln \frac{1}{\delta} \\
 \underset{ \delta \rightarrow 0}{\rightarrow} \: &\mathcal{G}_1 + \ln M' + c \\
= \:\:\: &\mathcal{G}_1 + \mathcal{G}_2 + \mathcal{N}(0, 4 \ln 2) + \ln \tilde{X}_2 + \ln \tilde{X}_3  + c.
\end{align*}
All the random variables appearing above are independent, $\mathcal{G}_1$ and $\mathcal{G}_2$ are two independent Gumbel laws, and $c$ is a non-universal real constant that depends on the regularization procedure. We have also used the fact that $\ln \tilde{X}_1 \overset{law}{=} \mathcal{G}_2$.  

Lastly we briefly mention that in the case of the interval it is also possible to see the GMC measure as the limit of the characteristic polynomial of random Hermitian matrices, the connection in this case was established in \cite{BWW}. The main result of \cite{BWW} is that for suitable random Hermitian matrices $H_N$, the quantity
\begin{equation*}
\frac{\vert \det(H_N -x) \vert^{\gamma}}{\mathbb{E}\vert \det(H_N -x) \vert^{\gamma}} dx
\end{equation*}
converges in law to the GMC measure on the unit interval $[0,1]$.\footnote{Actually in \cite{BWW} the limiting GMC measure is defined on $[-1,1]$ but of course by a change of variable we can write everything on $[0,1]$.} Therefore the same applications as the ones given in \cite{remy} hold and in particular one can conjecture that the following convergence in law holds:
\begin{align*}
\max_{x \in [0,1]} \ln \vert &\det(H_N -x)\vert -  \ln N + \frac{3}{4}\ln \ln N\\
 &\underset{ N \rightarrow \infty}{\rightarrow}   \: \mathcal{G}_1 + \mathcal{G}_2 + \mathcal{N}(0, 4 \ln 2) + \ln \tilde{X}_2 + \ln \tilde{X}_3  + c.
\end{align*}
This conjecture first appeared in \cite{FyoSimm} although it was written on $[-1,1]$ instead of $[0,1]$.

\section{The shift equations on $a$ and $p$}\label{sec_shift}

To prove Theorem \ref{main_result} we proceed in two steps. We first completely determine the dependence of $M(\gamma,p,a,b)$ on the parameters $a$ and $b$, see the result of Proposition \ref{prop_a} just below. We are then left with an unknown function $C(p)$ of $p$ (and $\gamma$) and give its value in Proposition \ref{prop_p}. Throughout this section we extensively use the fact that $U(t)$ and $\tilde{U}(t)$ are solutions of the hypergeometric equations of Proposition \ref{BPZ} proven in section \ref{section DE}.

\subsection{The shifts in $a$}
The goal of this subsection is to prove the shift equations \eqref{shift_equa1}, \eqref{shift_equa2} on $a$ and $b$ to completely determine the dependence of $M(\gamma,p,a,b)$ on these two parameters. By symmetry we will write everything only for $a$. We will thus prove that:
\begin{proposition}\label{prop_a} For $\gamma \in (0,2)$ and $a,b,p$ satisfying the bounds \eqref{bounds}, $M(\gamma, p, a,b )$ is given by the expression,
\begin{small}
\begin{equation}\label{equation M(a,b)}
 C(p) \frac{\Gamma_{\frac{\gamma}{2}}( \frac{2}{\gamma}(a+1 ) -(p-1) \frac{\gamma}{2}  )
\Gamma_{\frac{\gamma}{2}}( \frac{2}{\gamma}(b+1 ) -(p-1) \frac{\gamma}{2}  ) \Gamma_{\frac{\gamma}{2}}(
\frac{2}{\gamma}(a+b+2 ) -(p-2) \frac{\gamma}{2}  )  }{\Gamma_{\frac{\gamma}{2}}( \frac{2}{\gamma}(a+1 ) +
\frac{\gamma}{2}  ) \Gamma_{\frac{\gamma}{2}}( \frac{2}{\gamma}(b+1 ) + \frac{\gamma}{2}  )
\Gamma_{\frac{\gamma}{2}}( \frac{2}{\gamma}(a+b+2 ) -(2p-2) \frac{\gamma}{2}  )  
},
\end{equation}
\end{small}
where $C(p)$ is the function that contains the remaining dependence on $p$ (and $\gamma$). It will be computed in subsection \ref{subsection shift p}.
\end{proposition}

\noindent
$\diamond$  \textit{The $+ \frac{\gamma^2}{4}$ shift equation} \\ 
Here we start with the first auxiliary function, for $\gamma \in (0,2)$ and $a, b, p$ satisfying \eqref{bounds}: 
\begin{equation}
U(t) =   \mathbb{E}[ (\int_{0}^{1} (x-t)^{\frac{\gamma^2}{4}}  x^{a}(1-x)^{b} e^{\frac{\gamma}{2} X( x) } d x)^p ]. 
\end{equation}
From the result of Proposition \ref{BPZ}, $U(t)$ is solution to a hypergeometric equation. As explained in appendix \ref{sec_special} we can write the solutions of this hypergeometric equation for $t \in (- \infty,0)$ in two different bases, one corresponding to an expansion in powers of $\vert t \vert$ and one to an expansion in power of $\vert t \vert^{-1}$. Since the solution space is a two-dimensional real vector space, each basis will be parametrized by two real constants. Let $C_1,C_2$ and $D_1,D_2$ stand for these constants. The theory of hypergeometric equations then gives an explicit change of basis formula \eqref{hpy1} linking $C_1,C_2$ and $D_1,D_2$. Thus we can write, when $A-B$ and $C$ are not integers,
\begin{align}
U(t) &= C_1 F(A,B,C,t) \label{equation 2.3} \\ 
&+ C_2 |t|^{1 -C} F(1 + A-C, 1 +B - C, 2 -C, t) \nonumber \\
&= D_1 |t|^{-A}F(A,1+A-C,1+A-B,t^{-1}) \\ 
 &+ D_2 |t|^{-B} F(B, 1 +B - C, 1 +B - A, t^{-1}), \nonumber
\end{align}
where $F$ is the hypergeometric function.
We recall that the parameters $A,B,C$ are given by:
\begin{equation}
A= -\frac{p\gamma^2}{4}, B=-(a+b+1)-(2-p)\frac{\gamma^2}{4}, C = -a-\frac{\gamma^2}{4}.
\end{equation}
The values of $A,B,C$ left out corresponding to $A-B$ or $C$ being integers will be recovered at the level of the shift equation \eqref{shift equation 1} by continuity. The idea is now to identify the constants $C_1, C_2, D_1, D_2$ by performing asymptotic expansions on $U(t)$. Two of the above constants are easily obtained by evaluating $U(t)$ in $t=0$ and by taking the limit $t \rightarrow - \infty$:
\begin{align}
 C_1  &= M( \gamma, p, a + \frac{\gamma^2}{4}, b ), \\
 D_1 &= M(\gamma, p, a, b  ).
\end{align}
By performing a more detailed asymptotic expansion in $t \rightarrow - \infty$ we claim that:
\begin{equation}
D_2=0.
\end{equation}
We sketch a short proof. For $t<-2$ (arbitrary) and $x\in [0,1]$,
$$(x-t)^{\frac{\gamma^2}{4}}-|t|^{\frac{\gamma^2}{4}} \le c |t|^{\frac{\gamma^2}{4}-1},$$
for some constant $c>0$.
By interpolating, for $t<-2$,
\begin{small}
\begin{align*}
|U(t)-D_1 |t|^{\frac{p\gamma^2}{4} } |&= \Big| \E[(\int_{0}^{1} \big(u(x-t)^{\frac{\gamma^2}{4}}+(1-u)|t|^{\frac{\gamma^2}{4}}\big)  x^{a}(1-x)^{b} e^{\frac{\gamma}{2} X( x) } d x)^{p} ]\,|_{u=1}\\
&\quad -\E[(\int_{0}^{1} \big(u(x-t)^{\frac{\gamma^2}{4}}+(1-u)|t|^{\frac{\gamma^2}{4}}\big)  x^{a}(1-x)^{b} e^{\frac{\gamma}{2} X( x) } d x)^{p} ]\,|_{u=0} \Big| \\
 &\le |p|
 \int_0^1 dx_1 \big((x_1-t)^{\frac{\gamma^2}{4}}-|t|^{\frac{\gamma^2}{4}}\big)x_1^a(1-x_1)^b\\
  &\quad \times \Big(\E[(\int_{0}^{1} \frac{(x-t)^{\frac{\gamma^2}{4}}  x^{a}(1-x)^{b}}{|x_1-x|^{\frac{\gamma^2}{2}}} e^{\frac{\gamma}{2} X( x) } d x)^{p-1} ] \\
 &\quad \: \: +\E[(\int_{0}^{1} \frac{|t|^{\frac{\gamma^2}{4}}  x^{a}(1-x)^{b}}{|x_1-x|^{\frac{\gamma^2}{2}}} e^{\frac{\gamma}{2} X( x) } d x)^{p-1} ]\Big)\\
 &\le c'|t|^{\frac{p\gamma^2}{4}-1}M(\gamma,p,a,b)  
 \underset{t\to -\infty}{=} O(|t|^{\frac{p\gamma^2}{4}-1}),
\end{align*}
\end{small}
where in both steps we have used the Girsanov theorem (see appendix \ref{sec_th}) and $c'>0$ is some constant. However, by using the bound \eqref{bounds} over $p$:
\begin{equation}
(-A)-(-B) = -(a+b+1+(2-2p)\frac{\gamma^2}{4})<1.
\end{equation}
This implies that $D_2=0$. We then use the following identity coming from the theory of hypergeometric functions \eqref{hpy1}:
\begin{equation}
C_1 = \frac{\Gamma(1-C) \Gamma(A- B+1) }{\Gamma(A-C+1) \Gamma(1-B)} D_1. 
\end{equation}
This leads to the first shift equation \eqref{shift_equa1}:
\begin{equation}\label{shift equation 1}
\frac{M(\gamma, p, a + \frac{\gamma^2}{4} ,b )}{M(\gamma, p, a,b )}  =  
\frac{\Gamma(1 + a + \frac{\gamma^2}{4} ) \Gamma(2 + a + b - (2p -2
)\frac{\gamma^2}{4}  )  }{\Gamma(1 + a -(p-1) \frac{\gamma^2}{4} ) \Gamma(2 + a + b
- (p -2 )\frac{\gamma^2}{4}  ) } .
\end{equation}

\noindent
$\diamond$  \textit{The $+1$ shift equation} \\
We now write everything with the second auxiliary function, for $\gamma \in (0,2)$ and $a, b, p$ satisfying \eqref{bounds}: 
\begin{equation}
\tilde{U}(t) =   \mathbb{E}[ (\int_{0}^{1} (x-t)  x^{a}(1-x)^{b} e^{\frac{\gamma}{2} X( x) } d x)^p ] .
\end{equation}
Again we write the solutions of the hypergeometric equation around $t =0_-$ and $t = -\infty$, when $\tilde{C}$ and $\tilde{A} - \tilde{B}$ are not integers,
\begin{align}
\tilde{U}(t) &= \tilde{C}_1 F(\tilde{A},\tilde{B},\tilde{C},t)\\ \nonumber
 &+ \tilde{C}_2 |t|^{1 -\tilde{C}} F(1 + \tilde{A}-\tilde{C}, 1 +\tilde{B} - \tilde{C}, 2 -\tilde{C}, t) \\ 
&= \tilde{D}_1 |t|^{-\tilde{A}}F(\tilde{A},1+\tilde{A}-\tilde{C},1+\tilde{A}-\tilde{B},t^{-1}) \\ \nonumber
&+ \tilde{D}_2 |t|^{-\tilde{B}} F(\tilde{B}, 1 +\tilde{B} - \tilde{C}, 1 +\tilde{B} - \tilde{A}, t^{-1}).
\end{align}
As before we have introduced four real constants $\tilde{C}_1, \tilde{C}_2, \tilde{D}_1, \tilde{D}_2$ and $\tilde{A}, \tilde{B}, \tilde{C}$ are given by:
\begin{equation}
\tilde{A}= -p, \tilde{B}=-\frac{4}{\gamma^2}(a+b+2)+p-1, \tilde{C} = -\frac{4}{\gamma^2}(a+1).
\end{equation}
Two of our constants are again easily obtained,
\begin{align}
\tilde{C}_1  &= M( \gamma, p, a + 1, b ), \\
\tilde{D}_1 &= M(\gamma, p, a, b  ), 
\end{align}
and we can proceed as previously to obtain:
\begin{equation}
\tilde{D}_2=0.
\end{equation}
The relation between $\tilde{C}_1$ and $\tilde{D}_1$ \eqref{hpy1} then leads to the shift equation \eqref{shift_equa2}:
\begin{equation}\label{shift equation 2}
\frac{M(\gamma, p, a + 1 ,b )}{M(\gamma, p, a,b )}  =  
\frac{\Gamma(\frac{4}{\gamma^2}(1 +a) +1 ) \Gamma( \frac{4}{\gamma^2
} (2 + a +b) - (2p-2)) }{\Gamma(\frac{4}{\gamma^2}(1 +a) -(p-1 )) \Gamma(
\frac{4}{\gamma^2
} (2 + a +b) - (p-2)) } .
\end{equation}
Therefore for $\frac{\gamma^2}{4} \notin \mathbb{Q}$, \eqref{shift equation 1} and \eqref{shift equation 2} prove the formula of Proposition \ref{prop_a}. The result for the other values of $\gamma$ follows from the well known fact that $\gamma \mapsto M(\gamma,p,a,b)$ is a continuous function.

\subsection{The shifts in $p$}\label{subsection shift p}
We now tackle the problem of determining two shift equations on $p$, \eqref{relation_c1} and \eqref{relation_c2}, to completely determine the function $C(p)$ of Proposition \ref{prop_a}. We will work only with $U(t)$. The idea is to perform a computation at the next order in the expressions of the previous subsection. This will give the desired result:
\begin{proposition}\label{prop_p} For $\gamma \in (0,2)$ and $ p < \frac{4}{\gamma^2}$:
\begin{equation}
C(p) =  \frac{(2 \pi)^p}{\Gamma( 1  -  \frac{\gamma^2}{4} ) ^p}  (\frac{2}{\gamma})^{p\frac{\gamma^2}{4}}  \frac{\Gamma_{\frac{\gamma}{2}} ( \frac{2}{\gamma} -p \frac{\gamma}{2} ) }{\Gamma_{\frac{\gamma}{2}}(\frac{2}{\gamma})}.
\end{equation}
\end{proposition}

\noindent
$\diamond$  \textit{The $+1$ shift equation} \\
Since we have completely determined the dependence of $M$ on $a,b$ by equation \eqref{equation M(a,b)} we are free to choose $a$ and $b$ as we wish. To find the next order in $t \to 0_-$, the most natural idea is to take $a$ such that $0<1-C=1+a+\frac{\gamma^2}{4}<1$, and then it suffices to study the equivalent of $U(t)-U(0)$ when $t\to 0_-$. For technical reasons this only gives the expression of $C_2$ when $\gamma < \sqrt{2}$. To obtain $C_2$ for all $\gamma \in (0,2)$, we will need to go one order further in the asymptotic expansion and we make the choice $0<a<1-\frac{\gamma^2}{4}$ and $b=0$. In this case, we have $p<\frac{4}{\gamma^2}$, $1<1-C<2$. We perform a Taylor expansion around $t=0_-$,
\begin{equation*}
U(t)=U(0)+tU'(0)+t^2\int_0^1 U''(tu) (1-u)du ,
\end{equation*}
with
\begin{small}
\begin{equation*} 
\begin{split}
U''(tu) &\overset{(\star)}{=}-\frac{p\gamma^2}{4}\int_{0}^{1}  dx_1 (x_1-tu)^{\frac{\gamma^2}{4}-1}x_1^a\frac{a}{x_1}
\E[ (\int_{0}^{1} \frac{(x-tu)^{\frac{\gamma^2}{4}}  x^{a}}{|x-x_1|^{\frac{\gamma^2}{2}}} e^{\frac{\gamma}{2} X( x) } d x)^{p-1}]\\
&=-\frac{p\gamma^2a}{4}|tu|^{-1+a+\frac{\gamma^2}{4}}\int_{0}^{-\frac{1}{tu}}  dy (y+1)^{\frac{\gamma^2}{4}-1}y^{a-1} \E[ (\int_{0}^{1} \frac{(x-tu)^{\frac{\gamma^2}{4}}  x^{a}}{|x+tuy|^{\frac{\gamma^2}{2}}} e^{\frac{\gamma}{2} X( x) } d x)^{p-1}].
\end{split}
\end{equation*}
\end{small}
$(\star)$ comes from multiple applications of the Girsanov theorem (see appendix \ref{sec_th}) and symmetrization tricks. One may refer to \eqref{equation U'' primary} where we calculate rigorously the derivatives of $U(t)$. Next we have the following bound for $y \in [0,-\frac{1}{tu}]$, $u\in [0,1]$, and $t\in [-1,0]$:
\begin{small}
\begin{equation*}
\begin{split}
 &\E[ (\int_{0}^{1} \frac{(x-tu)^{\frac{\gamma^2}{4}}  x^{a}}{|x+tuy|^{\frac{\gamma^2}{2}}} e^{\frac{\gamma}{2} X( x) } d x)^{p-1}] \\
 &\le \sup_{x_1\in [0,1]}\big\{ \E[ (\int_{0}^{1} \frac{x^{a+\frac{\gamma^2}{4}} }{|x-x_1|^{\frac{\gamma^2}{2}}} e^{\frac{\gamma}{2} X( x) } d x)^{p-1}]+\E[ (\int_{0}^{1} \frac{(x+1)^{\frac{\gamma^2}{4}}  x^{a}}{|x-x_1|^{\frac{\gamma^2}{2}}} e^{\frac{\gamma}{2} X( x) } d x)^{p-1}]\big\}\\
& <\infty.
 \end{split}
\end{equation*}
\end{small}
Then we get by dominant convergence that,
\begin{small}
\begin{equation*}
U''(tu) \stackrel{t\to 0_-}{\sim}-\frac{p\gamma^2a}{4}|tu|^{-1+a+\frac{\gamma^2}{4}}\int_{0}^{\infty}  dy (y+1)^{\frac{\gamma^2}{4}-1}y^{a-1} M(\gamma,p-1,a-\frac{\gamma^2}{4},0),
\end{equation*}
\end{small}
and again by dominant convergence:
\begin{small}
\begin{equation*}
\begin{split}
&U(t)-U(0)-tU'(0)\\
&=-\frac{p\gamma^2 a}{4} \frac{\Gamma(a+\frac{\gamma^2}{4})}{\Gamma(2+a+\frac{\gamma^2}{4})}|t|^{1+a+\frac{\gamma^2}{4}}\int_{0}^{\infty}  dy (y+1)^{\frac{\gamma^2}{4}-1}y^{a-1} M(\gamma,p-1,a-\frac{\gamma^2}{4},0)\\
&+o(|t|^{1+a+\frac{\gamma^2}{4}}).
\end{split}
\end{equation*}
\end{small}
The value of the integral above is given by (\ref{equation integral computation 2}). We arrive at the expression for $C_2$:
\begin{equation}
C_2 = p\frac{\Gamma(a+1)\Gamma(-a-\frac{\gamma^2}{4}-1)}{\Gamma(-\frac{\gamma^2}{4})} M(\gamma, p-1, a - \frac{\gamma^2}{4},0).
\end{equation}
The theory of hypergeometric equations \eqref{hpy1} gives this time the relation:
\begin{equation}
C_2 =  \frac{\Gamma(C-1) \Gamma(A- B+1) }{\Gamma(A) \Gamma(C-B)} D_1.
\end{equation}
By identifying the above two expressions of $C_2$, we get
\begin{small}
\begin{equation*}
M(\gamma, p-1, a -\frac{\gamma^2}{4},0) = \frac{\Gamma(1 + a -(p-1) \frac{\gamma^2}{4} ) \Gamma(2 + a 
- (p -2 )\frac{\gamma^2}{4}  ) }{\Gamma(1 +a ) \Gamma(2 + a - (2p -3
)\frac{\gamma^2}{4}  )  }M(\gamma,  p,a, 0  ).
\end{equation*}
\end{small}
By using the shift equation \eqref{shift equation 1} on $a$, we can drop the $-\frac{\gamma^2}{4}$ after $a$ in the expression $M(\gamma, p-1, a -\frac{\gamma^2}{4},0)$ and we obtain for $0<a<1-\frac{\gamma^2}{4}$ and $b= 0$,
\begin{footnotesize}
\begin{equation*}
\frac{M(\gamma, p,a, 0  )}{M(\gamma, p-1,a, 0  )}=\frac{\Gamma(1-\frac{p\gamma^2}{4})}{\Gamma(1-\frac{\gamma^2}{4})} \frac{\Gamma(1 + a -(p-1) \frac{\gamma^2}{4} ) \Gamma(1  -(p-1) \frac{\gamma^2}{4} )\Gamma(2 + a 
- (p -2 )\frac{\gamma^2}{4}  ) }{ \Gamma(2 + a  - (2p -3
)\frac{\gamma^2}{4}  )  \Gamma(2 + a 
- (2p -2 )\frac{\gamma^2}{4}  )}.
\end{equation*}
\end{footnotesize}
Combined with \eqref{equation M(a,b)}, this leads to a first relation on our constant $C(p)$, for $p<\frac{4}{\gamma^2}$,
\begin{equation}\label{shiftc}
\frac{C(p)}{C(p-1)}  = \sqrt{2 \pi}  (\frac{\gamma}{2})^{(p-1)\frac{\gamma^2}{4} -\frac{1}{2} }  \frac{\Gamma(1 -p\frac{\gamma^2}{4} )   }{ \Gamma( 1  -  \frac{\gamma^2}{4} )  }.
\end{equation}
Reversely, \eqref{shiftc} and \eqref{equation M(a,b)} show that for all $a,b,p$ satisfying the bounds \eqref{bounds}:
\begin{footnotesize}
\begin{align}\label{shift equation 3}
& \frac{M(\gamma, p,a, b  )}{M(\gamma, p-1,a, b  )}= \\ \nonumber
&\frac{\Gamma(1-\frac{p\gamma^2}{4})}{\Gamma(1-\frac{\gamma^2}{4})} \frac{\Gamma(1 + a -(p-1) \frac{\gamma^2}{4} ) \Gamma(1 + b -(p-1) \frac{\gamma^2}{4} )\Gamma(2 + a + b
- (p -2 )\frac{\gamma^2}{4}  ) }{ \Gamma(2 + a + b - (2p -3
)\frac{\gamma^2}{4}  )  \Gamma(2 + a + b
- (2p -2 )\frac{\gamma^2}{4}  )}.
\end{align}
\end{footnotesize}

\noindent
$\diamond$  \textit{The $+ \frac{4}{\gamma^2}$ shift equation} \\
Since the relation \eqref{shiftc} is not enough to completely determine the function $C(p)$, we seek another relation on $C(p)$ that is not predicted by the Selberg integral. The techniques of this subsection are a little more involved, they lead to a relation between $C(p)$ and $C(p - \frac{4}{\gamma^2})$.   Again we can pick $a$ and $b$ as we wish so we choose $b=0$ and $-1-\frac{\gamma^2}{4}<a<-1-\frac{\gamma^2}{4}+a_0$ where $a_0>0$ is a constant introduced in lemma \ref{lemma fusion} of appendix \ref{section fusion}. The asymptotic in $t \rightarrow 0_-$ of the following quantity is then given by the lemma \ref{lemma fusion},
\begin{small}
\begin{equation*}
\begin{split}
&\mathbb{E}[ (\int_{0}^{1} (x-t)^{\frac{\gamma^2}{4}} \, x^{a} \,e^{\frac{\gamma}{2} X( x) } d x)^p ]  - \mathbb{E}[ (\int_{0}^{1} x^{a+\frac{\gamma^2}{4}} \,e^{\frac{\gamma}{2} X( x) } d x)^p ]  \\
&= g(\gamma,a)\frac{\Gamma(-p+ 1+\frac{4}{\gamma^2}(a+1))}{\Gamma(-p)}|t|^{1+a+\frac{\gamma^2}{4}} M(\gamma,p-1-\frac{4}{\gamma^2}(a+1),-2-a-\frac{\gamma^2}{4},0)\\
&+o( |t|^{1 +a+\frac{\gamma^2}{4}}),
\end{split}
\end{equation*}
\end{small}
where $g(\gamma,a)$ is a real function that only depends on $\gamma$ and $a$. Comparing with the expansion \eqref{equation 2.3}, we have:
\begin{equation}
C_2 =g(\gamma,a)\frac{\Gamma(-p+ 1+\frac{4}{\gamma^2}(a+1))}{\Gamma(-p)} M(\gamma,p-1-\frac{4}{\gamma^2}(a+1),-2-a-\frac{\gamma^2}{4},0).
\end{equation}
With the identity \eqref{hpy1} coming from hypergeometric equations:
\begin{align*}
C_2 &=  \frac{\Gamma(C-1) \Gamma(A- B+1) }{\Gamma(A) \Gamma(C-B)} D_1\\
 &= \frac{\Gamma(-1-a-\frac{\gamma^2}{4}) \Gamma( 2+a-(2p-2)\frac{\gamma^2}{4})}{\Gamma(-p\frac{\gamma^2}{4}) \Gamma(1-(p-1)\frac{\gamma^2}{4})} M(\gamma, p,a,0).
\end{align*}
Comparing the above two expressions of $C_2$ yields:
\begin{footnotesize}
\begin{align}\label{equation 2.25}
&g(\gamma,a)=\\ \nonumber
&\frac{ M(\gamma, p,a,0)}{M(\gamma,p-1-\frac{4}{\gamma^2}(a+1),-2-a-\frac{\gamma^2}{4},0)}\frac{\Gamma(-p)\Gamma(-1-a-\frac{\gamma^2}{4}) \Gamma(2+a-(2p-2)\frac{\gamma^2}{4})}{\Gamma(-p+ 1+\frac{4}{\gamma^2}(a+1))\Gamma(-p\frac{\gamma^2}{4}) \Gamma(1-(p-1)\frac{\gamma^2}{4})}.
\end{align}
\end{footnotesize}
A crucial remark is that from \eqref{equation M(a,b)} and analycity of the function $\Gamma_{\gamma}$, $M(\gamma,p,a,b)$ is analytic in $a,b$. Thus the right hand side of \eqref{equation 2.25} is analytic in $a$. We can then do analytic continuation simultaneously for both sides in the above equation. This shows that the expression of the right hand side does not depend on $p$ not only for  $-1-\frac{\gamma^2}{4}<a<-1-\frac{\gamma^2}{4}+a_0$ but for all appropriate $a$ where the expression is well-defined, i.e. $-1-\frac{\gamma^2}{4}<a<-1$.

In the following computations $f(\gamma)$ stands for a real function depending only on $\gamma$ and we will use the abuse of notation that it could be a different function of $\gamma$ every time it appears. Consider the case where $\frac{4}{k+1}<\gamma^2<\frac{4}{k}$ for a $k \in \mathbb{N}^*$. For this range of $\gamma$ we make the special choice $a= -\frac{(k+1)\gamma^2}{4}$ and thus the bounds $-1-\frac{\gamma^2}{4}<a<-1$ on $a$ are satisfied. In the previous paragraph we have shown that for $a = -\frac{(k+1)\gamma^2}{4}$:
\begin{equation}\label{equation 2.26}
\frac{ M(\gamma, p,-\frac{(k+1)\gamma^2}{4},0)}{M(\gamma,p-\frac{4}{\gamma^2}+k,\frac{k\gamma^2}{4}-2,0)}= f(\gamma)\frac{\Gamma(\frac{4}{\gamma^2}-k-p)\Gamma(-p\frac{\gamma^2}{4}) \Gamma(1-(p-1)\frac{\gamma^2}{4})}{\Gamma(-p)\Gamma(\frac{k\gamma^2}{4}-1) \Gamma(2-(2p+k-1)\frac{\gamma^2}{4})}.
\end{equation}
By the shift equations \eqref{shift equation 1} and \eqref{shift equation 2}:
\begin{footnotesize}
\begin{equation*}
\begin{split}
\frac{ M(\gamma, p-\frac{4}{\gamma^2}+k, \frac{k\gamma^2}{4}-2,0) }{M(\gamma,p-\frac{4}{\gamma^2}+k,-\frac{(k+1)\gamma^2}{4},0)} &= f(\gamma)\prod_{j=0}^1 \frac{\Gamma(j\frac{4}{\gamma^2}+1-p) \Gamma((1+j)\frac{4}{\gamma^2}+2-p)}{ \Gamma((2+j)\frac{4}{\gamma^2}-k+2-2p)} \\
 &\times \prod_{i=0}^{2k}\frac{\Gamma(4-(2p+3k-i-1)\frac{\gamma^2}{4})}{\Gamma(2-(p+2k-i)\frac{\gamma^2}{4}) \Gamma(3-(p+2k-i-1)\frac{\gamma^2}{4})}.
\end{split}
\end{equation*}
\end{footnotesize}
Then by \eqref{shift equation 3},
\begin{footnotesize}
\begin{equation*}
\begin{split}
&\frac{ M(\gamma, p-\frac{4}{\gamma^2}+k, -\frac{(k+1)\gamma^2}{4},0) }{M(\gamma,p-\frac{4}{\gamma^2},-\frac{(k+1)\gamma^2}{4},0)}\\
 &=f(\gamma)\prod_{i=0}^{k-1} \frac{\Gamma(2-(p+k+i+1)\frac{\gamma^2}{4}) \Gamma(2-(p+i)\frac{\gamma^2}{4}) \Gamma(2-(p+1+i)\frac{\gamma^2}{4}) \Gamma(3-(p+k+i)\frac{\gamma^2}{4}) }{\Gamma(4-(2p+k+2i)\frac{\gamma^2}{4}) \Gamma(4-(2p+k+2i+1)\frac{\gamma^2}{4}) },
\end{split}
\end{equation*}
\end{footnotesize}
and the product of the above two equations gives:
\begin{footnotesize}
\begin{equation*}
\begin{split}
\frac{ M(\gamma, p-\frac{4}{\gamma^2}+k, \frac{k\gamma^2}{4}-2,0) }{M(\gamma,p-\frac{4}{\gamma^2},-\frac{(k+1)\gamma^2}{4},0)} &= f(\gamma)\frac{\Gamma(4-(2p+k-1)\frac{\gamma^2}{4})}{\Gamma(3-(p-1)\frac{\gamma^2}{4})\Gamma(3-p\frac{\gamma^2}{4}) \prod_{i=0}^{k-2} (2-(p+1+i)\frac{\gamma^2}{4})} \\
&\times\prod_{j=0}^1 \frac{\Gamma(j\frac{4}{\gamma^2}+1-p) \Gamma((1+j)\frac{4}{\gamma^2}+2-p)}{ \Gamma((2+j)\frac{4}{\gamma^2}-k+2-2p)}.
\end{split}
\end{equation*}
\end{footnotesize}
Combining this relation with the previous shift equations \eqref{equation 2.26}:
\begin{footnotesize}
\begin{equation*}
\begin{split}
&\frac{ M(\gamma, p, -\frac{(k+1)\gamma^2}{4}, 0) }{M(\gamma,p-\frac{4}{\gamma^2},-\frac{(k+1)\gamma^2}{4},0)}\\
 &= f(\gamma)\frac{\Gamma(\frac{4}{\gamma^2}-k-p)\Gamma(-p\frac{\gamma^2}{4}) \Gamma(1-(p-1)\frac{\gamma^2}{4})\Gamma(4-(2p+k-1)\frac{\gamma^2}{4})}{\Gamma(-p)\Gamma(\frac{k\gamma^2}{4}-1) \Gamma(2-(2p+k-1)\frac{\gamma^2}{4})\Gamma(3-(p-1)\frac{\gamma^2}{4})} \\
 &\times \frac{\Gamma(1-p) \Gamma(\frac{4}{\gamma^2}+1-p) \Gamma(\frac{4}{\gamma^2}+2-p) \Gamma(\frac{8}{\gamma^2}+2-p)}{ \Gamma(3-p\frac{\gamma^2}{4})\prod_{i=0}^{k-2} (2-(p+1+i)\frac{\gamma^2}{4}) \Gamma(\frac{8}{\gamma^2}-k+2-2p) \Gamma(\frac{12}{\gamma^2}-k+2-2p)}\\
 &=f(\gamma) \frac{ \Gamma(-p\frac{\gamma^2}{4}) \Gamma(1-(p-1)\frac{\gamma^2}{4}) \Gamma(4-(2p+k-1)\frac{\gamma^2}{4}) \Gamma(1-p)  }{\Gamma(3-p\frac{\gamma^2}{4}) \Gamma(3-(p-1)\frac{\gamma^2}{4}) \Gamma(2-(2p+k-1)\frac{\gamma^2}{4}) \Gamma(-p)  } \\
 &\times \frac{\Gamma(\frac{8}{\gamma^2}+2-p) \Gamma(\frac{4}{\gamma^2}+2-p) \Gamma(\frac{4}{\gamma^2} -k -p ) \Gamma(\frac{4}{\gamma^2} + 1 -p) }{\prod_{i=0}^{k-2} (\frac{8}{\gamma^2}-(p+1+i)) \Gamma( \frac{12}{\gamma^2} -k+2-2p ) \Gamma(\frac{8}{\gamma^2} -k+2-2p)}\\
 &=f(\gamma) \Gamma(\frac{4}{\gamma^2}-p) \frac{\Gamma(\frac{4}{\gamma^2} -k -p ) \Gamma(\frac{4}{\gamma^2} + 1 -p) \Gamma(\frac{8}{\gamma^2} -k+1-p) }{\Gamma( \frac{12}{\gamma^2} -k+1-2p ) \Gamma(\frac{8}{\gamma^2} -k+1-2p)}.
\end{split}
\end{equation*}
\end{footnotesize}
By \eqref{equation M(a,b)}, the same ratio of $M$ can also be written as,
\begin{small}
\begin{align*}
&\frac{M(\gamma,p,-\frac{(k+1)\gamma^2}{4},0)}{  M(\gamma,p-\frac{4}{\gamma^2},-\frac{(k+1)\gamma^2}{4},0)}\\
  &= \frac{C(p)}{C(p -\frac{4}{\gamma^2})} f(\gamma)(\frac{\gamma}{2} )^{ p} \frac{\Gamma(\frac{4}{\gamma^2} -k -p ) \Gamma(\frac{4}{\gamma^2} + 1 -p) \Gamma(\frac{8}{\gamma^2} -k+1-p) }{\Gamma( \frac{12}{\gamma^2} -k+1-2p ) \Gamma(\frac{8}{\gamma^2} -k+1-2p)},
\end{align*}
\end{small}
thus we obtain for $\frac{4}{k+1}<\gamma^2<\frac{4}{k}$:
\begin{equation}\label{shiftc2}
\frac{C(p)}{C(p-\frac{4}{\gamma^2})} = f(\gamma) (\frac{\gamma} {2})^{-p}\Gamma(\frac{4}{\gamma^2}-p).
\end{equation}
This proves the second shift equation \eqref{relation_c2} on $C(p)$. Then for every fixed $\gamma $ such that $\frac{4}{\gamma^2} \notin \mathbb{Q} $ both shift equations \eqref{shiftc} and \eqref{shiftc2} completely determine the value $C(p)$ up to a constant $c_{\gamma}$ of $\gamma$. To see this, take another continuous function $\mathfrak{C}(p)$ that satisfies both shift equations \eqref{shiftc} and \eqref{shiftc2}. Then the ratio $\mathfrak{R}(p):=\frac{C(p)}{\mathfrak{C}(p)}$ is a 1-periodic and $\frac{4}{\gamma^2}$-periodic continuous function.
Combining this with the fact that $\frac{4}{\gamma^2} \notin \mathbb{Q} $ implies that the ratio $\mathfrak{R}(p)$ is constant and $C(p)$ is determined up to a constant $c_\gamma$ of $\gamma$ by the two shift equations on $p$.

The constant $c_{\gamma}$ is then evaluated by choosing $p=0$ and by using the known value $M(\gamma,0,a,b) =1$. Thus we arrive at the formula of Proposition \ref{prop_p}. Finally by the continuity of $\gamma \rightarrow M(\gamma,p,a,b)$, we can extend the formula to the values of $\gamma$ that were left out. This completes the proof of Proposition \ref{prop_p}.

\section{Proof of the differential equations}\label{section DE}
We now move to the proof of Proposition \ref{BPZ}. In order to show that $U(t)$ and $\tilde{U}(t)$ satisfy these differential equations we will need to introduce a regularization procedure. We will work with two small parameters $\delta >0 $ and $\epsilon >0$ which will be sent to $0$ at the appropriate places in the proof. The first parameter $\delta$ controls the cut-off procedure used to smooth $X$. A convenient smoothing procedure can be written by seeing $X$ as the restriction of the centered Gaussian field defined on the disk $\mathbb{D}+(\frac{1}{2},0)$, i.e. the unit disk centered in $(\frac{1}{2},0)$. $X$ still has a covariance given by:
\begin{equation}
\E[X(x)X(y)]=2\ln\frac{1}{\vert x-y\vert}.
\end{equation}
Then for any smooth function  $\theta \in \mathcal{C}^{\infty}([0,\infty),\mathbb{R}_+)$ with support in $[0,1]$ and satisfying $\int_0^{\infty} \theta =  \frac{1}{\pi}$, we write $\theta_{\delta}:=\frac{1}{\delta^2}\theta(\frac{|\cdot|^2}{\delta^2})$ and define the regularized field $X_{\delta}:= X*\theta_{\delta}$.
Similarly we introduce: 
\begin{equation}
\frac{1}{(x)_{\delta}}:=\int_{\mathbb{C}}\int_{\mathbb{C}}\frac{1}{x+y_1+y_2}\theta_{\delta}(y_1)\theta_{\delta}(y_2)d^2y_1d^2y_2.
\end{equation}
This quantity will appear when we take the derivative of $\E[X_{\delta}(x)X_{\delta}(y)]$. Now since we have the singularities $x^a$ and $(1-x)^b$ that appear in $U(t)$ and $\tilde{U}(t)$, we will also need to restrict the integration from $[0,1]$ to the smaller interval $[\epsilon, 1-\epsilon]$ for some small $\epsilon$ that will be sent to $0$. Finally we introduce some more compact notations for various expressions that depend on both $\delta$ and $\epsilon$:
\begin{small}
\begin{align*}
G_{\delta}(x,y) &:=\E[X_{\delta}(x)X_{\delta}(y)] \\
D(x;t) &:=(x-t)^{\frac{\gamma^2}{4} } 
x^{a}(1-x)^{b}  \\
 U_{\epsilon,\delta}(t)  &:=\E[(\int_{\epsilon}^{1-\epsilon}D(x;t)e^{\frac{\gamma}{2}X_{\delta}(x)}dx)^p] \\
 V^{(1)}_{\epsilon,\delta}(x_1;t) &:= \E[(\int_{\epsilon}^{1-\epsilon}D(x;t) e^{\frac{\gamma}{2}X_{\delta}(x)+\frac{\gamma^2}{4}G_{\delta}(x,x_1)}dx)^{p-1}] \\
 V^{(2)}_{\epsilon,\delta}(x_1,x_2;t) &:= \E[(\int_{\epsilon}^{1-\epsilon}D(x;t) e^{\frac{\gamma}{2}X_{\delta}(x)+\frac{\gamma^2}{4}(G_{\delta}(x,x_1)+G_{\delta}(x,x_2))}dx)^{p-2}] \\
 E_{0,\epsilon,\delta}(t) &:=D(\epsilon;t)V^{(1)}_{\epsilon,\delta}(\epsilon;t) \\
  E_{1,\epsilon,\delta} (t) &:= D(1-\epsilon;t)V^{(1)}_{\epsilon,\delta}(1-\epsilon;t).
\end{align*}
\end{small}
The terms $V_{\epsilon, \delta}^{(1)}$ and $V_{\epsilon, \delta}^{(2)}$ will appear when we compute respectively the first and second order derivatives of $U_{\epsilon, \delta}$. The terms $E_{0,\epsilon, \delta}$ and $E_{1,\epsilon, \delta}$ are the boundary terms of the integration by parts performed below. We will also use $U_{\epsilon}(t)$, $ V^{(1)}_{\epsilon}(x_1;t) $, $ V^{(2)}_{\epsilon}(x_1,x_2;t)$, $  E_{0,\epsilon}(t)$, $ E_{1,\epsilon} (t) $ for the limit of the above quantities as $\delta$ goes to $0$.
\begin{proof}
First we prove the equation for $U(t)$. We recall the definition,
\begin{small}
\begin{equation}
U(t) =   \mathbb{E}[ (\int_{0}^{1} (x-t)^{\frac{\gamma^2}{4}}  x^{a}(1-x)^{b} e^{\frac{\gamma}{2} X( x) } d x)^p ],
\end{equation}
\end{small}
and we calculate the derivatives with the help of the Girsanov theorem of appendix \ref{sec_th}:
\begin{small}
\begin{align*}
U_{\epsilon,\delta}'(t)=& p\int_{\epsilon}^{1-\epsilon} dx_1\, \partial_t D(x_1;t) V^{(1)}_{\epsilon,\delta}(x_1;t)\\
=&-p\int_{\epsilon}^{1-\epsilon} dx_1 \partial_{x_1}((x_1-t)^{\frac{\gamma^2}{4}}) x_1^{a}(1-x_1)^{b}V^{(1)}_{\epsilon,\delta}(x_1;t)\\
=&-p \Big( E_{1,\epsilon,\delta}(t)-E_{0,\epsilon,\delta}(t)- \int_{\epsilon}^{1-\epsilon}  dx_1 D(x_1;t)V^{(1)}_{\epsilon,\delta}(x_1;t)(\frac{a}{x_1}-\frac{b}{1-x_1})\\
& - \int_{\epsilon}^{1-\epsilon}  dx_1 D(x_1;t)\,\partial_{x_1} V^{(1)}_{\epsilon,\delta}(x_1;t)\Big).
\end{align*}
\end{small}
We claim that the last term in the sum equals zero. Indeed,
\begin{small}
\begin{align*}
\int_{\epsilon}^{1-\epsilon}  &dx_1 D(x_1;t)\,\partial_{x_1} V^{(1)}_{\epsilon,\delta}(x_1;t)\\
=&(p-1)\frac{\gamma^2}{2}\int_{\epsilon}^{1-\epsilon} \int_{\epsilon}^{1-\epsilon} dx_1 dx_2 \frac{D(x_1;t)D(x_2;t)}{(x_2-x_1)_{\delta}}e^{\frac{\gamma^2}{4}G_{\delta}(x_2,x_1)}
V^{(2)}_{\epsilon,\delta}(x_1,x_2;t)\\
=& 0 \qquad \text{by symmetry.}
\end{align*}
\end{small}
Thus, by sending $\delta$ to $0$, 
\begin{small}
\begin{equation}\label{equation U'}
U_{\epsilon}'(t)=-p \Big( E_{1,\epsilon}(t)-E_{0,\epsilon}(t)- \int_{\epsilon}^{1-\epsilon}  dx_1 D(x_1;t)V^{(1)}_{\epsilon}(x_1;t)(\frac{a}{x_1}-\frac{b}{1-x_1})\Big).
\end{equation}
\end{small}
In the same spirit, we calculate:
\begin{small}
\begin{equation*}
\begin{split}
U_{\epsilon,\delta}''(t) &= \frac{p\gamma^2}{4}\Big[ -\int_{\epsilon}^{1-\epsilon} dx_1\partial_t\big(\frac{D(x_1;t)}{(x_1-t)}\big)V^{(1)}_{\epsilon,\delta}(x_1;t)\\
&+ \frac{(p-1)\gamma^2}{4}  \int_{\epsilon}^{1-\epsilon} dx_1\int_{\epsilon}^{1-\epsilon} dx_2  \frac{D(x_1;t)D(x_2;t)}{(x_1-t)(x_2-t)} e^{\frac{\gamma^2}{4}G_{\delta}(x_2,x_1)}V^{(2)}_{\epsilon,\delta}(x_1,x_2;t) \Big].
\end{split}
\end{equation*}
\end{small}
An integration by parts gives:
\begin{small}
\begin{align*}
-&\int_{\epsilon}^{1-\epsilon} dx_1\partial_t \big(\frac{D(x_1;t)}{(x_1-t)}\big)V^{(1)}_{\epsilon,\delta}(x_1;t)\\
=& \int_{\epsilon}^{1-\epsilon}dx_1 \partial_{x_1}(\frac{(x_1-t)^{\frac{\gamma^2}{4}}}{x_1-t}) x_1^{a}(1-x_1)^{b}V^{(1)}_{\epsilon,\delta}(x_1;t)\\
=&\frac{1}{1-t-\epsilon}E_{1,\epsilon,\delta}(t)+\frac{1}{t-\epsilon}E_{0,\epsilon,\delta}(t)\\ 
&-\int_{\epsilon}^{1-\epsilon}  dx_1 D(x_1;t)V^{(1)}_{\epsilon,\delta}(x_1;t)\frac{1}{x_1-t}(\frac{a}{x_1}-\frac{b}{1-x_1})\\
&-\frac{(p-1)\gamma^2}{2}  \int_{\epsilon}^{1-\epsilon} dx_1\int_{\epsilon}^{1-\epsilon} dx_2   \frac{D(x_1;t)D(x_2;t)}{(x_1-t)(x_2-x_1)_{\delta}}e^{\frac{\gamma^2}{4}G_{\delta}(x_2,x_1)} V^{(2)}_{\epsilon,\delta}(x_1,x_2;t).
\end{align*}
\end{small}
By symmetry of the expression under the exchange of $x_1$ and $x_2$,
\begin{small}
\begin{align*}
&\frac{(p-1)\gamma^2}{2}  \int_{\epsilon}^{1-\epsilon} dx_1\int_{\epsilon}^{1-\epsilon} dx_2   \frac{D(x_1;t)D(x_2;t)}{(x_1-t)(x_2-x_1)_{\delta}}e^{\frac{\gamma^2}{4}G_{\delta}(x_2,x_1)}V^{(2)}_{\epsilon,\delta}(x_1,x_2;t)\\
=&\frac{(p-1)\gamma^2}{4}   \int_{\epsilon}^{1-\epsilon} dx_1\int_{\epsilon}^{1-\epsilon} dx_2   D(x_1;t)D(x_2;t)e^{\frac{\gamma^2}{4}G_{\delta}(x_2,x_1)} \\
& \times (\frac{1}{(x_1-t)(x_2-x_1)_{\delta}} +\frac{1}{(x_2-t)(x_1-x_2)_{\delta}})
V^{(2)}_{\epsilon,\delta}(x_1,x_2;t) \\
=&\frac{(p-1)\gamma^2}{4}  \int_{\epsilon}^{1-\epsilon} dx_1\int_{\epsilon}^{1-\epsilon} dx_2   \frac{D(x_1;t)D(x_2;t)}{(x_1-t)(x_2-t)}\frac{x_2-x_1}{(x_2-x_1)_{\delta}}e^{\frac{\gamma^2}{4}G_{\delta}(x_2,x_1)}V^{(2)}_{\epsilon,\delta}(x_1,x_2;t)  .
\end{align*}
\end{small}
Since $\frac{x_2-x_1}{(x_2-x_1)_{\delta}}  \le c$ for some constant $c>0$ independent of $\delta$, by sending $\delta$ to $0$,
\begin{small}
\begin{align}\label{equation U'' primary}
U_{\epsilon}''(t) =\frac{p\gamma^2}{4} &\Big(\frac{1}{1-t-\epsilon}E_{1,\epsilon}(t)+\frac{1}{t-\epsilon}E_{0,\epsilon}(t) \\
 &-\int_{\epsilon}^{1-\epsilon}  dx_1 D(x_1;t)V^{(1)}_{\epsilon}(x_1;t)\frac{1}{x_1-t}(\frac{a}{x_1}-\frac{b}{1-x_1})\Big). \nonumber
\end{align}
\end{small}
A further calculation shows that,
\begin{small}
\begin{align*}
&\int_{\epsilon}^{1-\epsilon}  dx_1 D(x_1;t)V^{(1)}_{\epsilon}(x_1;t)\frac{1}{x_1-t}(\frac{a}{x_1}-\frac{b}{1-x_1})\\
&=\int_{\epsilon}^{1-\epsilon}  dx_1 D(x_1;t)V^{(1)}_{\epsilon}(x_1;t)\big(\frac{a}{t}(\frac{1}{x_1-t}-\frac{1}{x_1})-\frac{b}{1-t}(\frac{1}{x_1-t}+\frac{1}{1-x_1})\big)\\
&=-\int_{\epsilon}^{1-\epsilon}  dx_1 D(x_1;t)V^{(1)}_{\epsilon}(x_1;t)(\frac{a}{tx_1}+\frac{b}{(1-t)(1-x_1)})-\frac{4}{p\gamma^2}(\frac{a}{t}-\frac{b}{1-t})U_{\epsilon}'(t),
\end{align*}
\end{small}
and as a consequence,
\begin{small}
\begin{equation}\label{equation U''}
\begin{split}
U_{\epsilon}''(t) =&\frac{p\gamma^2}{4} \Big(\frac{1}{1-t-\epsilon}E_{1,\epsilon}(t-\epsilon)+\frac{1}{t}E_{0,\epsilon}(t)\\ 
&\quad \quad+\int_{\epsilon}^{1-\epsilon}  dx_1 D(x_1;t)V^{(1)}_{\epsilon}(x_1;t)(\frac{a}{tx_1}+\frac{b}{(1-t)(1-x_1)})\Big)\\
&+(\frac{a}{t}-\frac{b}{1-t})U_{\epsilon}'(t).
\end{split}
\end{equation}
\end{small}
We can also write $U_{\epsilon,\delta}(t)$ in a similar form, by doing an integration by parts:
\begin{small}
\begin{equation*}
\begin{split}
(1&-t-\epsilon)E_{1,\epsilon,\delta}(t)+(t-\epsilon)E_{0,\epsilon,\delta}(t)\\
&-\int_{\epsilon}^{1-\epsilon} dx_1 D(x_1;t)V^{(1)}_{\epsilon,\delta}(x_1;t)(x_1-t)(\frac{a}{x_1}-\frac{b}{1-x_1})\\
=& (1+\frac{\gamma^2}{4})\int_{\epsilon}^{1-\epsilon} dx_1D(x_1;t)V^{(1)}_{\epsilon,\delta}(x_1;t)\\
&+(p-1)\frac{\gamma^2}{2}\int_{\epsilon}^{1-\epsilon}\int_{\epsilon}^{1-\epsilon} dx_1dx_2 D(x_1;t)D(x_2;t)e^{\frac{\gamma^2}{4}G_{\delta}(x_2,x_1)}\frac{x_1-t}{(x_2-x_1)_{\delta}}V^{(2)}_{\epsilon,\delta}(x_1,x_2;t)\\
=&(1+\frac{\gamma^2}{4})\int_{\epsilon}^{1-\epsilon} dx_1D(x_1;t)V^{(1)}_{\epsilon,\delta}(x_1;t)\\
&-(p-1)\frac{\gamma^2}{4}\int_{\epsilon}^{1-\epsilon} dx_1\int_{\epsilon}^{1-\epsilon} dx_2   D(x_1;t)D(x_2;t)e^{\frac{\gamma^2}{4}G_{\delta}(x_2,x_1)}\frac{x_2-x_1}{(x_2-x_1)_{\delta}}V^{(2)}_{\epsilon,\delta}(x_1,x_2;t).
\end{split}
\end{equation*}
\end{small}
By sending $\delta$ to $0$ and by applying the Girsanov theorem of appendix \ref{sec_th}, we obtain,
\begin{small}
\begin{align*}
-(B+a+b) U_{\epsilon}(t)= &(1-t-\epsilon)E_{1,\epsilon}(t)+(t-\epsilon)E_{0,\epsilon}(t)\\
&-\int_{\epsilon}^{1-\epsilon} dx_1D(x_1;t)V^{(1)}_{\epsilon}(x_1;t)(x_1-t)(\frac{a}{x_1}-\frac{b}{1-x_1}),
\end{align*}
\end{small}
where we recall that $B = -(a+b+1) -(2-p)\frac{\gamma^2}{4}$. We also note that,
\begin{small}
\begin{align*}
\int_{\epsilon}^{1-\epsilon} & dx_1 D(x_1;t)V^{(1)}_{\epsilon}(x_1;t)(x_1-t)(\frac{a}{x_1}-\frac{b}{1-x_1})\\
=&(a+b)U_{\epsilon,\delta}(t)-\int_{\epsilon}^{1-\epsilon} dx_1 D(x_1;t)V^{(1)}_{\epsilon}(x_1;t)(\frac{at}{x_1}+\frac{b(1-t)}{1-x_1}),
\end{align*}
\end{small}
and hence,
\begin{small}
\begin{equation}\label{equation U}
-BU_{\epsilon}(t)=(1-t-\epsilon)E_{1,\epsilon}(t)+(t-\epsilon)E_{0,\epsilon}(t)+\int_{\epsilon}^{1-\epsilon} dx_1 D(x_1;t)V^{(1)}_{\epsilon}(x_1;t)(\frac{at}{x_1}+\frac{b(1-t)}{1-x_1}).
\end{equation}
\end{small}
Combining this with the expressions for $U'_{\epsilon}$ and $U''_{\epsilon}$, equations \eqref{equation U'} and \eqref{equation U''},
\begin{small}
\begin{align*}
U_{\epsilon}'(t)=&-p \Big( E_{1,\epsilon}(t)-E_{0,\epsilon}(t)- \int_{\epsilon}^{1-\epsilon}  dx_1 D(x_1;t)V^{(1)}_{\epsilon}(x_1;t)(\frac{a}{x_1}-\frac{b}{1-x_1})\Big),\\
U_{\epsilon}''(t) =&\frac{p\gamma^2}{4} \Big(\frac{1}{1-t-\epsilon}E_{1,\epsilon}(t)+\frac{1}{t-\epsilon}E_{0,\epsilon}(t)\\
 & \quad \quad+\int_{\epsilon}^{1-\epsilon}  dx_1D(x_1;t)V^{(1)}_{\epsilon}(x_1;t)(\frac{a}{tx_1}+\frac{b}{(1-t)(1-x_1)})\Big)\\
&+(\frac{a}{t}-\frac{b}{1-t})U_{\epsilon}'(t),
\end{align*}
\end{small}
we finally arrive at:
\begin{align}
&t(1-t)U_{\epsilon}''(t)+(C-(A+B+1)t)U_{\epsilon}'(t)-ABU_{\epsilon}(t) \\ 
& \quad \quad \quad =\epsilon(1-\epsilon)\frac{p\gamma^2}{4} (\frac{1}{1-t-\epsilon}E_{1,\epsilon}(t)+\frac{1}{t-\epsilon}E_{0,\epsilon}(t)). \nonumber
\end{align}
From this expression we see that the last thing we need to check is that as $\epsilon$ goes to zero the right hand side of the above expression converges to $0$ in a suitable sense. Indeed we will prove that, for $t$ in a fixed compact set $K \subseteq (-\infty,0)$, $\epsilon E_{1,\epsilon}(t)$ and $\epsilon E_{0,\epsilon}(t)$ converge uniformly to $0$ for a well chosen sequence of $\epsilon$. Let us consider $\epsilon E_{0,\epsilon}(t)$ as $\epsilon E_{1,\epsilon}(t)$ can be treated in a similar fashion: 
\begin{equation*}
\epsilon E_{0,\epsilon}(t) = (\epsilon-t)^{\frac{\gamma^2}{4}}\epsilon^{a+1}(1-\epsilon)^b\E[(\int_{\epsilon}^{1-\epsilon}\frac{(x-t)^{\frac{\gamma^2}{4}}x^a(1-x)^b}{|x-\epsilon|^{\frac{\gamma^2}{2}}}e^{\frac{\gamma}{2}X(x)}dx)^{p-1}].
\end{equation*}
In the following we will discuss three disjoint cases based on the value of $a$. They are $a > -1 + \frac{\gamma^2}{4}$, $ -1 < a \leq -1 + \frac{\gamma^2}{4}$, and $ - 1 - \frac{\gamma^2}{4} < a \leq -1$.\\
\noindent
\textbf{i)} $a>-1+\frac{\gamma^2}{4}$\\
This is the simplest case as we have for $\epsilon$ sufficiently small and for some $c_0>0$,
\begin{align*}
 \epsilon E_{0,\epsilon}(t) &\le c_0 \epsilon^{a+1}(1-\epsilon)^b\E[(\int_{0}^{1}\frac{x^{a}(1-x)^b}{|x-\epsilon|^{\frac{\gamma^2}{2}}}e^{\frac{\gamma}{2}X(x)}dx)^{p-1}]\\
  &\overset{\epsilon \to 0}{\sim} c_0 \epsilon^{a+1} M(\gamma,p,a-\frac{\gamma^2}{2},b),
\end{align*}
which converges to 0 as $ \epsilon \rightarrow 0$ uniformly over $t\in K$.\\
\noindent
\textbf{ii)} $-1< a \le -1+\frac{\gamma^2}{4}$.\\
In this case we have $p-1<1$ and $\epsilon^{a+1} \underset{\epsilon \rightarrow 0}{\longrightarrow} 0$. If $p-1\le 0$, $$\E[(\int_{\epsilon}^{1-\epsilon}\frac{(x-t)^{\frac{\gamma^2}{4}}x^a(1-x)^b}{|x-\epsilon|^{\frac{\gamma^2}{2}}}e^{\frac{\gamma}{2}X(x)}dx)^{p-1}]$$ is uniformly bounded thus it is immediate to obtain the convergence to $0$. Hence it suffices to consider the case $0<p-1 < 1$. We choose $\epsilon_N =\frac{1}{2^N}$. Using the sub-additivity of the function $x\mapsto x^{p-1}$, we have for some $c_0, c' >0$ independent of $K$: 
\begin{align*}
\epsilon_N E_{0,\epsilon_N}(t) \le &c_0 \epsilon_N^{a+1} \E[(\int_{\epsilon_{N}}^{\frac{1}{2}}\frac{x^a}{|x-\epsilon_{N}|^{\frac{\gamma^2}{2}}}e^{\frac{\gamma}{2}X(x)}dx)^{p-1}]+ c' \epsilon_N^{a+1} \\
\le & c_0 \epsilon_N^{a+1} \sum_{n=1}^{N-1}\E[(\int_{\epsilon_{n+1}}^{\epsilon_n}\frac{x^a}{|x-\epsilon_{n+1}|^{\frac{\gamma^2}{2}}}e^{\frac{\gamma}{2}X(x)}dx)^{p-1}]+ c' \epsilon_N^{a+1} .
\end{align*}
Then by the scaling property of GMC,
\begin{align*}
\E[(&\int_{\epsilon_{n+1}}^{\epsilon_n}\frac{x^a}{|x-\epsilon_{n+1}|^{\frac{\gamma^2}{2}}}e^{\frac{\gamma}{2}X(x)}dx)^{p-1}] \\
= &2^{\frac{\gamma^2}{4}(p-1)(p-2)-(a-\frac{\gamma^2}{2}+1)(p-1)}\E[(\int_{\epsilon_{n}}^{\epsilon_{n-1}}\frac{u^a}{|u-\epsilon_{n}|^{\frac{\gamma^2}{2}}}e^{\frac{\gamma}{2}X(u)}du)^{p-1}]\\
=&2^{\frac{\gamma^2}{4}p^2-(\frac{\gamma^2}{4}+a+1)p+a+1}\E[(\int_{\epsilon_{n}}^{\epsilon_{n-1}}\frac{u^a}{|u-\epsilon_{n}|^{\frac{\gamma^2}{2}}}e^{\frac{\gamma}{2}X(u)}du)^{p-1}].
\end{align*}
We can deduce that,
\begin{small}
\begin{align*}
\epsilon_N E_{0,\epsilon_N}(t)  \le \: &c_1 2^{-N(a+1)}2^{(N-1)(\frac{\gamma^2}{4}p^2-(\frac{\gamma^2}{4}+a+1)p+a+1)}\E[(\int_{\frac{1}{4}}^{\frac{1}{2}}\frac{x^{a}}{|x-\frac{1}{4}|^{\frac{\gamma^2}{2}}}e^{\frac{\gamma}{2}X(x)}dx)^{p-1}]\\
 &+c' \epsilon_N^{a+1} \\
\le &  \:c 2^{N(\frac{\gamma^2}{4}p-\frac{\gamma^2}{4}-a-1)p}+c' \epsilon_N^{a+1}  \overset{N \to \infty}{\longrightarrow} 0,
\end{align*}
\end{small}
for some constants $c_1,c,c'>0$. The convergence holds since $p>0$ and $\frac{\gamma^2}{4}p-\frac{\gamma^2}{4}-a-1<0$ (this inequality comes from \eqref{bounds}), and it holds uniformly over $t$ in $K$.\\
\textbf{iii)} $-1-\frac{\gamma^2}{4} <a \le -1$\\
In this case $p-1<0$ so we are always dealing with negative moments. This implies that for $t$ in $K$, we can bound $\epsilon E_{0,\epsilon}(t)$ by,
\begin{equation*}
\epsilon E_{0,\epsilon}(t) \le c_0 \epsilon^{a+1} \E[(\int_{\epsilon}^{\frac{1}{2}} x^{a-\frac{\gamma^2}{2}}e^{\frac{\gamma}{2}X(x)}dx)^{p-1}],
\end{equation*}
simply by restricting the integral over $[\epsilon, 1 -\epsilon]$ to $[\epsilon, 1/2]$. An estimation of the resulting GMC moment is given by lemma \ref{lemma estimation GMC} in appendix \ref{section estimation GMC}. For $\epsilon$ sufficiently small, there exists a constant $c >0$ such that,
\begin{align*}
&\E[(\int_{\epsilon}^{\frac{1}{2}} x^{a-\frac{\gamma^2}{2}}e^{\frac{\gamma}{2}X(x)}dx)^{p-1}]\\
\le  &
\begin{cases}
c \,\epsilon^{(\frac{\gamma}{4}-\frac{1}{\gamma}(a+1))^2}, &1+a+\frac{\gamma^2}{4}-\frac{p\gamma^2}{2} >0\\
c \,\epsilon^{(p-1)(1+a-\frac{\gamma^2}{4})-\frac{(p-1)^2\gamma^2}{4}}, &1+a+\frac{\gamma^2}{4}-\frac{p\gamma^2}{2} \le 0
\end{cases}
\end{align*}
This suffices to show the convergence to $0$ of $\epsilon E_{0,\epsilon}(t)$. 

Indeed, in the first case, a basic inequality shows that $(\frac{\gamma}{4}-\frac{1}{\gamma}(a+1))^2 \ge -(a+1) $ with equality when $-(a+1)=\frac{\gamma^2}{4}$. Since the condition cannot be satisfied, we have the strict inequality. In the second case where $1+a+\frac{\gamma^2}{4}-\frac{p\gamma^2}{2}  \le 0$, we can easily show that under this condition together with the bound \eqref{bounds} for $p$, $(p-1)(1+a-\frac{\gamma^2}{4})-\frac{(p-1)^2\gamma^2}{4}>-(a+1)$. Hence in both cases, 
$\epsilon E_{0,\epsilon}(t) {\longrightarrow}0$, where the convergence is again uniform over $t$ in $K$.

Combining the cases (i), (ii) and (iii), we have proven the differential equation \ref{ODE U} in the weak sense (in the sense of distributions). Since it is a hypoelliptic equation (the dominant operator is a Laplacian) with analytic coefficients, $U(t)$ is analytic and the equation holds in the strong sense.\\
Let us now briefly mention the case of $\tilde{U}(t)$. In a similar manner, we calculate,
\begin{small}
\begin{align*}
-\tilde{B}\tilde{U}(t)=&\frac{4}{\gamma^2}\Big((1-t-\epsilon)\tilde{E}_{1,\epsilon}(t)+(t-\epsilon)\tilde{E}_{0,\epsilon}(t)\\ 
&\quad \quad +\int_{\epsilon}^{1-\epsilon} dx_1 \tilde{D}(x_1;t)\tilde{V}^{(1)}_{\epsilon}(x_1;t)(\frac{at}{x_1}+\frac{b(1-t)}{1-x_1})\Big),\\
\tilde{U}_{\epsilon}'(t)=&-p \Big( \tilde{E}_{1,\epsilon}(t)-\tilde{E}_{0,\epsilon}(t)- \int_{\epsilon}^{1-\epsilon}  dx_1 \tilde{D}(x_1;t)\tilde{V}^{(1)}_{\epsilon}(x_1;t)(\frac{a}{x_1}-\frac{b}{1-x_1})\Big),\\
\tilde{U}_{\epsilon}''(t) =&\frac{4p}{\gamma^2} \Big(\frac{1}{1-t-\epsilon}\tilde{E}_{1,\epsilon}(t)+\frac{1}{t-\epsilon}\tilde{E}_{0,\epsilon}(t)\\ 
&\quad \quad+\int_{\epsilon}^{1-\epsilon}  dx_1 \tilde{D}(x_1;t)\tilde{V}^{(1)}_{\epsilon}(x_1;t)(\frac{a}{tx_1}+\frac{b}{(1-t)(1-x_1)})\Big)\\
&+\frac{4}{\gamma^2}(\frac{a}{t}-\frac{b}{1-t})\tilde{U}_{\epsilon}'(t),
\end{align*}
\end{small}
where $\tilde{D}(x;t) :=(x-t)\,x^{a}(1-x)^{b}$ and where $\tilde{V}^{(1)}_{\epsilon}(x_1;t)$, $\tilde{E}_{0,\epsilon}(t)$, $\tilde{E}_{1,\epsilon}(t)$ are defined as functions of $\tilde{D}(x;t)$, the same as their definitions without the tilde.
We verify easily that,
\begin{align}
t(1-t)\tilde{U}_{\epsilon}''(t)+&(\tilde{C}-(\tilde{A}+\tilde{B}+1)t)\tilde{U}_{\epsilon}'(t)-\tilde{A}\tilde{B}\tilde{U}_{\epsilon}(t) \\
&=\epsilon(1-\epsilon)\frac{p\gamma^2}{4} (\frac{1}{1-t-\epsilon}\tilde{E}_{1,\epsilon}(t)+\frac{1}{t-\epsilon}\tilde{E}_{0,\epsilon}(t)), \nonumber
\end{align}
and the right hand side of the above expression converges again to zero uniformly for $t$ in any compact set of $(- \infty, 0)$, which finishes the proof of the Proposition \ref{BPZ}.
\end{proof}
One may wonder if other differential equations can be obtained for similar observables. If instead of $U(t)$ and $\tilde{U}(t)$ one introduces the more general function
\begin{equation}
t \rightarrow  \mathbb{E}[ (\int_{0}^{1} (x-t)^{\chi}  x^{a}(1-x)^{b} e^{\frac{\gamma}{2} X( x) } d x)^p ]
\end{equation}
for some arbitrary real number $\chi$, then this function will be solution to a second order differential equation if and only if $\chi = \frac{\gamma^2}{4} $ or $\chi =1$ (except for some special cases where ``non-trivial" relations hold for instance for $p=0$). This fact can be obtained by similar computations as the ones performed above. On the other hand conformal field theory predicts that differential equations of any order are expected to be verified by suitable observables although it is not clear to us at this stage what information can be extracted from these higher order differential equations.

\appendix

\section{Proof of the lemmas on GMC}

\subsection{Reminder on some useful theorems}\label{sec_th}

We recall some theorems in probability that we will use without further justification. In the following, $D$ is a compact subset of $\mathbb{R}^d$. 
\begin{theorem}[Girsanov theorem]
Let $(Z(x))_{x\in D}$ be a continuous centered Gaussian process and $Z$ a Gaussian variable which belongs to the $L^2$ closure of the vector space spanned by $(Z(x))_{x\in D}$. Let $F$ be a real continuous bounded function from $\mathcal{C}(D,\mathbb{R})$ to $\mathbb{R}$. Then we have the following identity:
\begin{equation}\label{equation Girsanov}
\E[e^{Z-\frac{\E[Z^2]}{2}}F((Z(x))_{x\in D})] = \E[F((Z(x)+\E[Z(x)Z])_{x\in D})].
\end{equation}
\end{theorem}
When applied to our case, although the log-correlated field $X$ is not a continuous Gaussian process, we can still make the arguments rigorous by using a regularization procedure. Let us illustrate the idea by a simple example that is used in section \ref{section DE}.  We introduce three cut-off parameters, $\delta$ to smooth the log-correlated field $X$, $\epsilon$ to avoid the singularities in $0$ and $1$, and $A>0$ to apply \eqref{equation Girsanov} to a bounded functional $F$. Hence the following computation:
\begin{align*}
&\E[(\int_0^1 x^a(1-x)^b e^{\frac{\gamma}{2}X(x)}dx)^{p}]\\
&=\lim_{\epsilon\to 0}\lim_{\delta\to 0}  \lim_{A \to +\infty}  \int_{\epsilon}^{1-\epsilon} dx_1 \,x_1^a(1-x_1)^b \E\Bigg[  \mathbf{1}_{\|X_{\delta}\|_{\infty} \le A}e^{\frac{\gamma}{2}X_{\delta}(x_1)-\frac{\gamma^2}{8}\E[X_{\delta}(x_1)^2]}\\
&\qquad \qquad \qquad \qquad \qquad \qquad (\int_{\epsilon}^{1-\epsilon} x^a(1-x)^b e^{\frac{\gamma}{2}X_{\delta}(x)-\frac{\gamma^2}{8}\E[X_{\delta}(x)^2]}dx)^{p-1} \Bigg]\\
&\overset{\eqref{equation Girsanov}}{=}\lim_{\epsilon\to 0}\lim_{\delta\to 0}\lim_{A \to +\infty} \int_{\epsilon}^{1-\epsilon} dx_1 \,x_1^a(1-x_1)^b \E\Bigg[ \mathbf{1}_{\|X_{\delta}\|_{\infty} \le A}\\
& \quad \quad \quad \quad\quad  (\int_{\epsilon}^{1-\epsilon}x^a(1-x)^b e^{\frac{\gamma}{2}X_{\delta}(x)+\frac{\gamma^2}{4}\E[X_{\delta}(x)X_{\delta}(x_1)]-\frac{\gamma^2}{8}\E[X_{\delta}(x)^2]}dx)^{p-1} \Bigg]\\
&=\int_{0}^{1} dx_1 \,x_1^a(1-x_1)^b \E[(\int_{0}^{1} \frac{x^a(1-x)^b}{|x_1-x|^{\frac{\gamma^2}{2}}} e^{\frac{\gamma}{2}X(x)}dx)^{p-1}].
\end{align*}
The next theorem is a comparison result due to Kahane \cite{Kah}:
\begin{theorem}[Convexity inequality]\label{theo convex inequality}
Let $(Z_1(x))_{x\in D}$, $(Z_2(x))_{x\in D}$ be two continuous centered Gaussian processes such that for all $x,y\in D$:
\begin{equation*}
\E[Z_1(x)Z_1(y)] \le \E[Z_2(x)Z_2(y)].
\end{equation*}
Then for all convex function (resp. concave) $F$ with at most polynomial growth at infinity, and $\sigma$ a positive finite measure over $D$,
\begin{equation}
\E[F(\int_D e^{Z_1(x)-\frac{1}{2}\E[Z_1(x)^2]} \sigma(dx))] \le (\text{ resp.} \ge \,) \, \E[F(\int_D e^{Z_2(x)-\frac{1}{2}\E[Z_2(x)^2]} \sigma(dx))].
\end{equation}
\end{theorem}
To apply this theorem to log-correlated fields, one needs again to use a regularization procedure. Finally, we provide the Williams decomposition theorem, see for instance \cite{Williams}:
\begin{theorem}\label{theo Williams}
Let $(B_s-vs)_{s\ge 0}$ be a Brownian motion with negative drift, i.e. $v>0$ and let $M = \sup_{s\ge 0}(B_s-vs)$. Then conditionally on $M$ the law of the path $(B_s-vs)_{s\ge 0}$ is given by the joining of two independent paths:\\
1) A Brownian motion $(B^1_s+vs)_{0\le s \le \tau_M}$ with positive drift $v$ run until its hitting time $\tau_M$ of $M$.\\
2) $(M+B^2_t-vt)_{t\ge 0}$ where $(B^2_t-vt)_{t\ge 0}$ is a Brownian motion with negative drift conditioned to stay negative.
\par Moreover, one has the following time reversal property for all $C>0$ (where $\tau_C$ denotes the hitting time of $C$),
\begin{equation}\label{equation 4.3}
(B^1_{\tau_C-s}+v(\tau_C-s)-C)_{0\le s\le \tau_C} \overset{\text{law}}{=} (\tilde{B}_s-vs)_{0\le s \le L_{-C}},
\end{equation}
where $(\tilde{B}_s-vs)_{s\ge 0}$ is a Brownian motion with drift $-v$ conditioned to stay negative and $L_{-C}$ is the last time $(\tilde{B}_s-vs)_{s\ge 0}$ hits $-C$.
\end{theorem}

\subsection{An estimate on GMC}\label{section estimation GMC}
We now move on to the proof of some technical lemmas required in the previous sections. Lemma \ref{lemma estimation GMC} written below will be used in section \ref{section DE} to show that the boundary terms obtained in the derivation of the differential equations converge to $0$. Just like in section \ref{section tail} for $s \ge 0$ we write $X(e^{-s/2}) = B_s+Y(e^{-s/2})$ where $B_s$ is a standard Brownian motion and $Y$ is an independent centered Gaussian field on $\mathbb{C}$ with covariance:
\begin{equation}\label{equation correlation Y}
\E[Y(x)Y(y)]=2\ln \frac{\vert x \vert \lor \vert y \vert}{|x-y|}.
\end{equation}
Denote the GMC measure associated to $Y(e^{-s/2})$ by $\mu_Y(ds):=e^{\frac{\gamma}{2}Y(e^{-s/2})}ds$. The goal of this subsection is to prove the following lemma:
\begin{lemma}\label{lemma estimation GMC}
 For  $q > 0$, $a < - 1 - \frac{\gamma^2}{4}$, and a fixed constant $A>0$, there exists $\epsilon_1<A$ sufficiently small such that for all $\epsilon\le \epsilon_1$,
\begin{equation}
\E[(\int_{\epsilon}^{A} x^{a}e^{\frac{\gamma}{2}X(x)}dx)^{-q}]\le  
\begin{cases}
c  \,\epsilon^{(\frac{\gamma}{4}+\frac{1}{\gamma}(a+1))^2} , & 1+a+\frac{\gamma^2}{4}+\frac{q\gamma^2}{2}  >0, \\
c \, \epsilon^{-q(1+a+\frac{\gamma^2}{4})-\frac{q^2\gamma^2}{4}}, &1+a+\frac{\gamma^2}{4}+\frac{q\gamma^2}{2}  \le 0
\end{cases}
\end{equation}
where $c>0$ is a constant that depends on $A,  \gamma, a$ and $q$.  
\end{lemma}
By using the decomposition described above, we can transform this lemma into another equivalent form,
\begin{equation*}
\begin{split}
  \E[(\int_{\epsilon}^{A} x^{a}e^{\frac{\gamma}{2}X(x)}dx)^{-q}]=& 2^{q} \E[(\int_{-2\ln A}^{-2\ln \epsilon}e^{\frac{\gamma}{2}(B_s-s(\frac{\gamma}{4}+\frac{1}{\gamma}(a+1)))}\mu_Y(ds))^{-q}]\\
=&2^{q} \E[(\int_{-2\ln A}^{-2\ln \epsilon}e^{\frac{\gamma}{2}(B_s+\alpha s)}\mu_Y(ds))^{-q}],
\end{split}
\end{equation*}
where again $(B_s)_{s\ge 0}$ is a standard Brownian motion independent from $Y$, and $\alpha = -\frac{\gamma}{4}-\frac{1}{\gamma}(a+1)$.
Therefore lemma \ref{lemma estimation GMC} is equivalent to the following lemma:
\begin{lemma}\label{lemma estimation r}
 For $q>0$, $\alpha>0$, a fixed constant $r_0$, there exists $r_1>r_0$ sufficiently large such that for all $r\ge r_1$,
\begin{equation}
\E[(\int_{r_0}^{r}e^{\frac{\gamma}{2}(B_s+\alpha s)}\mu_Y(ds))^{-q}] \le 
\begin{cases}
c\, e^{-\frac{\alpha^2}{2}r}, &\alpha-\frac{q\gamma}{2} < 0\\
c\, e^{(\frac{q^2\gamma^2}{8}-\frac{q\gamma \alpha}{2})r},  &\alpha-\frac{q\gamma}{2} \ge 0
\end{cases}
\end{equation}
where $c>0$ is a constant that depends on $r_0, \gamma, \alpha$ and $q$.
\end{lemma}
A similar result for 2d GMC has been proved in \cite{DOZZ1} (proposition 5.1). A slight difference is that in \cite{DOZZ1} the power $q$ depends on $a$.
\par  We start by proving three intermediate results. We denote $y_s = B_s+\alpha s$, and we introduce for $\beta\ge 1$ the stopping time $T_{\beta} = \inf \{s\ge 0, y_s = \beta-1\}$. Recall the density of $T_{\beta}$ for $\beta >1, u>0$: 
\begin{equation}
 \P(T_{\beta}\in (u,u+du)) = \frac{\beta-1}{\sqrt{2\pi}u^{3/2}}e^{-\frac{(\beta-1-\alpha u)^2}{2u}}du.
\end{equation}

\begin{lemma}\label{lemma sup}
For $\alpha,A>0$, we have:
\begin{equation}
\P(\sup_{s\le t}y_s\le A) \le e^{\alpha A-\frac{\alpha^2 t}{2}}.
\end{equation}
\end{lemma}
\begin{proof}
We know the density of $\sup_{s\le t}y_s$:
\begin{align*}
\P(\sup_{s\le t}y_s\le A) =\P(T_{A+1}\ge t)&= \frac{A}{\sqrt{2\pi}}\int_t^{\infty}\frac{e^{-\frac{(A-\alpha s)^2}{2s}}}{s^{3/2}}ds\\
 &\le \frac{A e^{\alpha A-\frac{\alpha^2 t}{2}}}{\sqrt{2\pi}}\int_0^{\infty}\frac{e^{-\frac{A^2}{2s}}}{s^{3/2}}ds =e^{\alpha A-\frac{\alpha^2 t}{2}}.
\end{align*}
\end{proof}
\begin{lemma}\label{lemma I}
We set  for $t>0$:
\begin{equation}
I(t) = \int_{t}^{t+1}e^{\frac{\gamma}{2}(y_s-y_t)}\mu_Y(ds).
\end{equation}
For $q>0$, we have the following inequality,
\begin{equation}
\E[I(t)^{-q}|y_{t+1}-y_t]\le c_1(e^{-\frac{\gamma}{2}q(y_{t+1}-y_t)}+1)\quad a.s.,
\end{equation}
where $c_1$ depends on $\gamma,q$.
\end{lemma}
\begin{proof}
Conditioning on $y_{t+1}-y_t=y$, $(B_s-B_t)_{t\le s \le t+1}$ has the law of a Brownian bridge between $0$ and $y-\alpha$. Hence it has the law of $(B'_s-sB'_{1} +s(y-\alpha))_{0\le s\le 1}$, where $B'$ is an independent Brownian motion. We have:
\begin{equation*}
\E[I(t)^{-q}|y_{t+1}-y_t=y] = \E[( \int_{0}^{1}e^{\frac{\gamma}{2}(B'_s-s B'_1 +sy)}\mu_Y(ds))^{-q}].
\end{equation*}
Notice that $e^{\frac{\gamma }{2}sy} \ge e^{\frac{\gamma }{2}y} \wedge 1$, and a classic result on the moments of Gaussian multiplicative chaos shows that,
$$\E[(\mu_Y([0,1]))^{-q}] < \infty,$$
thus:
\begin{align*}
\E[( \int_{0}^{1}e^{\frac{\gamma}{2}(B'_s-s B'_1)}\mu_Y(ds))^{-q}] &\le \E[e^{-\frac{q\gamma}{2}\inf_{0\le s \le 1}(B'_s-s B'_1)}] \E[(\mu_Y([0,1]))^{-q}]\\
&=:c_1 <\infty.
\end{align*}
We can now derive that:
\begin{equation*}
\E[I(t)^{-q}|y_{t+1}-y_t=y] \le c_1 (e^{-\frac{\gamma}{2}qy}\lor 1) \le c_1(e^{-\frac{\gamma}{2}qy}+1)\quad a.s.
\end{equation*}
\end{proof}
\begin{lemma}\label{lemma J}
Define for $\beta>1$, $\alpha>0$, $q>0$ and $r\ge 2$:
\begin{equation}
J_{r,\beta}: = \E[\frac{\mathbf{1}_{\{\sup_{s\in [0,r]}y_s \in [\beta -1,\beta]\}}}{(\int_{0}^{r}e^{\frac{\gamma}{2}y_s}\mu_Y(ds))^q}].
\end{equation}
Then there exists $c_2>0$ depending on $\gamma, \alpha, q$ such that:
\begin{equation}
 J_{r,\beta} \le c_2 e^{-\frac{\alpha^2}{2}r} e^{(\alpha-\frac{q\gamma}{2})\beta}.
\end{equation}
\end{lemma}
\begin{proof}
\begin{align*}
J_{r,\beta}&\le e^{-\frac{q\gamma (\beta-1)}{2}}\E[\mathbf{1}_{\{T_{\beta}\le r-1\}}\frac{\mathbf{1}_{\{\sup_{s\in [0,r]}y_s \in [\beta -1,\beta]\}}}{I(T_{\beta})^q}]\\
&+\E[\mathbf{1}_{\{T_{\beta}> r-1\}}\frac{\mathbf{1}_{\{\sup_{s\in [0,r]}y_s \in [\beta -1,\beta]\}}}{e^{\frac{q\gamma y_{r-1}}{2}}I(r-1)^q}]=:A+B.
\end{align*}
We first bound $A$. By using the strong Markov property of $(y_s)_{s\ge 0}$ with respect to $\mathcal{F}_{T_{\beta}+1}$:
\begin{align*}
A\le  &e^{-\frac{q\gamma (\beta-1)}{2}}\E[\mathbf{1}_{\{T_{\beta}+1\le r\}}I(T_{\beta})^{-q}\mathbf{1}_{\{\sup_{s\in [T_{\beta}+1,r]}y_s-y_{T_{\beta}+1} \le \beta-y_{T_{\beta}+1}\}}]\\
=&e^{-\frac{q\gamma (\beta-1)}{2}}\E[\mathbf{1}_{\{T_{\beta}+1\le r\}}I(T_{\beta})^{-q}\E[\mathbf{1}_{\{\sup_{s\in [0,r-T_{\beta}-1]}y_s' \le \beta-y_{T_{\beta}+1}\}}|\mathcal{F}_{T_{\beta}+1}]]\\
=&e^{-\frac{q\gamma (\beta-1)}{2}}\E\Big[\mathbf{1}_{\{T_{\beta}+1\le r\}}\E[I(T_{\beta})^{-q}|\mathcal{F}_{T_{\beta}},\beta -y_{T_{\beta}+1}]\\
&\times\E[\mathbf{1}_{\{\sup_{s\in [0,r-T_{\beta}-1]}y_s' \le \beta-y_{T_{\beta}+1}\}}|\mathcal{F}_{T_{\beta}},\beta -y_{T_{\beta}+1}]\Big].
\end{align*}
By lemma \ref{lemma I},
\begin{equation*}
\E[I(T_{\beta})^{-q}|\mathcal{F}_{T_{\beta}},\beta-y_{T_{\beta}+1}] \le c_1(e^{-\frac{\gamma}{2}q(y_{T_{\beta}+1}-\beta)}+1) \quad a.s.
\end{equation*}
By lemma \ref{lemma sup}, 
\begin{equation*}
\E[\mathbf{1}_{\{\sup_{s\in [0,r-T_{\beta}-1]}y_s' \le \beta-y_{T_{\beta}+1}\}}|\mathcal{F}_{T_{\beta}},\beta -y_{T_{\beta}+1}]\le e^{\alpha(\beta-y_{T_{\beta}+1})-\frac{\alpha^2 (r-T_{\beta}-1)}{2}}\quad a.s.
\end{equation*}
Therefore:
\begin{equation*}
A、\le c_1 e^{-\frac{q\gamma (\beta-1)}{2}}\E\Big[\mathbf{1}_{\{T_{\beta}+1\le r\}} (e^{-\frac{\gamma}{2}q(y_{T_{\beta}+1}-\beta)}+1)  e^{\alpha(\beta-y_{T_{\beta}+1})-\frac{\alpha^2 (r-T_{\beta}-1)}{2}}\Big].
\end{equation*}
Conditioning on $\mathcal{F}_{T_{\beta}}$, $y_{T_{\beta}+1}-\beta$ has the law of $N+\alpha$ where $N\sim \mathcal{N}(0,1)$. Hence,
\begin{equation*}
\begin{split}
A、\le & c_1 e^{-\frac{q\gamma (\beta-1)}{2}}\E[(e^{-\frac{\gamma}{2}q(N+\alpha)}+1)e^{-\alpha (N+\alpha)}]\E[\mathbf{1}_{\{T_{\beta}+1\le r\}} e^{-\frac{\alpha^2 (r-T_{\beta}-1)}{2}}]\\
=& c_1 e^{-\frac{q\gamma (\beta-1)}{2}} (e^{-\frac{\alpha^2}{2}+\frac{\gamma^2 q^2}{8}}+e^{-\frac{\alpha^2}{2}})\E[\mathbf{1}_{\{T_{\beta}+1\le r\}} e^{-\frac{\alpha^2 (r-T_{\beta}-1)}{2}}]\\
\le &  c_1 e^{-\frac{q\gamma (\beta-1)}{2}} (e^{\frac{\gamma^2 q^2}{8}}+1)e^{-\frac{\alpha^2 r}{2}}\E[\mathbf{1}_{\{T_{\beta}\le r-1\}} e^{\frac{\alpha^2 T_{\beta}}{2}}].
\end{split}
\end{equation*}
We calculate with the density of $T_{\beta}$:
\begin{equation*}
\begin{split}
\E[\mathbf{1}_{\{T_{\beta}\le r-1\}} e^{\frac{\alpha^2 T_{\beta}}{2}}] =& \int_0^{r-1}\frac{\beta-1}{\sqrt{2\pi}u^{3/2}}e^{-\frac{(\beta-1-\alpha u)^2}{2u}}e^{\frac{\alpha^2 u}{2}}du\\ 
&=e^{\alpha(\beta-1)}\sqrt{\frac{2}{\pi}}\int_{\frac{\beta-1}{\sqrt{r-1}}}^{\infty}e^{-\frac{x^2}{2}}dx\\
&\le    e^{\alpha(\beta-1)}   .
\end{split}
\end{equation*}
Combining the elements above we get,
\begin{equation}\label{equation A}
 A\le c_1' e^{-\frac{\alpha^2r}{2}}e^{(\alpha-\frac{q\gamma}{2})\beta },
\end{equation}
for some constant $c_1'>0$ of $\gamma, \alpha$ and $q$.
We proceed similarly for $B$, using again the Markov property:
\begin{small}
\begin{align*}
B = &\E\Big[ \mathbf{1}_{\{T_{\beta} > r-1\}} \frac{\mathbf{1}_{\{\sup_{s\in [r-1,r]}(y_s-y_{r-1}) \in [\beta -1 -y_{r-1},\beta-y_{r-1}]\}}}{e^{\frac{q\gamma y_{r-1}}{2}}I(r-1)^q}  \Big]\\
\le &\E\Big[ \mathbf{1}_{\{T_{\beta}> r-1\}} \frac{1}{e^{\frac{q\gamma }{2}(\beta-1-\sup_{s\in [r-1,r]}(y_s-y_{r-1}))}I(r-1)^q}  \Big]\\
= & e^{-\frac{q\gamma}{2}(\beta-1)}\P(T_{\beta} > r-1) \E\Big[  e^{\frac{q\gamma }{2}\sup_{s\in [r-1,r]}(y_s-y_{r-1})}I(r-1)^{-q}   \Big]
\end{align*}
\end{small}
We show that the expectation term can be easily bounded: let us denote $(y'_s)_s$ an independent process which has the same law as $(y_s)_s$,
\begin{align*}
&\E\Big[ e^{\frac{q\gamma }{2}\sup_{s\in [r-1,r]}(y_s-y_{r-1})}I(r-1)^{-q}  \Big]\\
\le & \E\Big[ e^{q\gamma\sup_{s\in [0,1]}y_s'} \Big]^{\frac{1}{2}} \E \Big[I(r-1)^{-2q} \Big]^{\frac{1}{2}}\\
\le &c_1^{\frac{1}{2}} \E\Big[ e^{q\gamma \sup_{s\in [0,1]}y_s'} \Big]^{\frac{1}{2}}\cdot   \E[e^{-\gamma q  y_1'}+1]^{\frac{1}{2}} ,
\end{align*}
where in the last inequality we have used lemma \ref{lemma I}.  We see that this whole expression is a constant that depends on $\gamma, \alpha$ and $q$. 

Now it suffices to compute:
\begin{align*}
\P(T_{\beta} >r-1)&=\int_{r-1}^{\infty} \frac{\beta-1}{\sqrt{2\pi}u^{3/2}}e^{-\frac{(\beta-1-\alpha u)^2}{2u}}du \\
 &\le \frac{\beta-1}{\sqrt{2 \pi}} e^{\alpha(\beta-1)-\frac{\alpha^2 (r-1)}{2}} \int_{r-1}^{\infty} u^{-3/2} e^{-\frac{(\beta-1)^2}{2u}} du \\
 & \le  e^{\alpha(\beta-1)-\frac{\alpha^2 (r-1)}{2} }     .
\end{align*}
Hence
\begin{equation}\label{equation B}
B \le  c_1'' e^{-\frac{\alpha^2r}{2}}e^{(\alpha-\frac{q\gamma}{2})\beta } .
\end{equation}

Equations \eqref{equation A} and \eqref{equation B} together finish the proof of the lemma.
\end{proof}
Now we can prove the main lemma:
\begin{proof}[Proof of lemma \ref{lemma estimation r}]
Define for $n \ge 1$:
\begin{equation}
M_n = \{\sup_{s\in [r_0,r]}(y_s-y_{r_0}) \in [n-1,n] \} .
\end{equation}
We can write,
\begin{align*}
 \E[(\int_{r_0}^{r}e^{\frac{\gamma}{2}y_s}\mu_Y(ds))^{-q}] =&e^{(\frac{q^2\gamma^2}{8}-\frac{q\gamma\alpha}{2})r_0} \sum_{n\ge 1}\E[\mathbf{1}_{M_n}(\int_{r_0}^{r}e^{\frac{\gamma}{2}(y_s-y_{r_0})}\mu_Y(ds))^{-q}]\\
 =& e^{(\frac{q^2\gamma^2}{8}-\frac{q\gamma\alpha}{2})r_0}\sum_{n\ge 1}J_{r-r_0,n},
\end{align*}
and by lemma \ref{lemma J}  when $r-r_0 \ge 2$: 
\begin{equation*}
J_{r-r_0,n}\le  c_2  e^{-\frac{\alpha^2 r}{2}}e^{(\alpha-\frac{q\gamma}{2})n } .
\end{equation*}
In the case where $\alpha-\frac{q\gamma}{2}<0$, it is then straightforward that there exists $c$ depending on $r_0, \gamma, \alpha, q$ such that:
\begin{align*}
\E[(\int_{r_0}^{r}e^{\frac{\gamma}{2}y_s}\mu_Y(ds))^{-q}] \le c\, e^{-\frac{\alpha^2r}{2}}.
\end{align*}
The other case where $\alpha - \frac{q\gamma}{2} \ge 0$ is actually very direct to prove, since we then have:
\begin{align*}
\E[(\int_{r_0}^{r}e^{\frac{\gamma}{2}y_s}\mu_Y(ds))^{-q}] \leq \E[e^{-\frac{q\gamma}{2}y_{r-1}}] \E[I(r-1)^{-q} ] \le c \,e^{(\frac{q^2\gamma^2}{8} - \frac{q\gamma\alpha}{2})r}.
\end{align*}
In the last inequality we have used the fact that $y_{r-1} = B_{r-1} + \alpha (r-1)$ and that $\E[I(r-1)^{-q} ] $ is a constant independent of $r$ that we can absorb in $c$. Notice this argument actually works whenever $\alpha > 0$. This finishes the proof of lemma \ref{lemma estimation GMC}.
\end{proof}

\subsection{Fusion estimation and the reflection coefficient}\label{section fusion}
In this subsection we will prove the asymptotic expansion result that is used in subsection \ref{subsection shift p} to obtain the shift equation \eqref{relation_c2} on $p$ with a shift $\frac{4}{\gamma^2}$. In this expansion will appear the reflection coefficient introduced in section \ref{section tail} which will also be discussed in the next subsection. Here we will thus show:
\begin{lemma}\label{lemma fusion}
For $-1-\frac{\gamma^2}{4}< a <-1-\frac{\gamma^2}{4}+a_0$  with $a_0>0$ a constant chosen small enough, $p<1+\frac{4}{\gamma^2}(a+1)$, as $t \rightarrow 0_-$,
\begin{small}
\begin{equation}
\begin{split}
U(t) &=M(\gamma,p,a+\frac{\gamma^2}{4},0)\\ 
&+   g(\gamma,a)\frac{\Gamma(-p+ 1+\frac{4}{\gamma^2}(a+1))}{\Gamma(-p)}|t|^{1+a+\frac{\gamma^2}{4}} M(\gamma,p-1-\frac{4}{\gamma^2}(a+1),-2-a-\frac{\gamma^2}{4},0)\\
&+o( |t|^{1 +a+\frac{\gamma^2}{4}}),
\end{split}
\end{equation}
\end{small}
where $g(\gamma,a)$ is defined as:
\begin{equation}
g(\gamma,a) = -\Gamma(-\frac{4}{\gamma^2}(a+1))\mathbb{E}[ (\frac{1}{2}\int_{-\infty}^{\infty}e^{\frac{\gamma}{2}\mathcal{B}_s^{\frac{\gamma}{4}+\frac{1}{\gamma}(a+1)}}\mu_{Y}(ds))^{ 1+\frac{4}{\gamma^2}(a+1)}].
\end{equation}
The process $\mathcal{B}^{\frac{\gamma}{4}+\frac{1}{\gamma}(a+1)}$ is defined by \eqref{mathcal B} and $\mu_Y(ds) = e^{\frac{\gamma}{2}Y(e^{-s/2})}ds$ is the notation introduced in section \ref{section estimation GMC}. 
\end{lemma}
Notice that in the expression of $g(\gamma,a)$ we recognize the reflection coefficient $\overline{R}_1^{\partial}(-\frac{2a}{\gamma})$ of section \ref{section tail}.  We emphasize that we only need the result for $a$ in a small open set, it is not necessary to obtain an explicit value for $a_0$.
\begin{remark}
From the conditions on $a$ and $p$ in the lemma, we have $-2-a-\frac{\gamma^2}{4}>-1-\frac{\gamma^2}{4}$ and $p-1-\frac{4}{\gamma^2}(a+1)<0$, thus the bounds \eqref{bounds} are satisfied and $M(\gamma,p-1-\frac{4}{\gamma^2}(a+1),-2-a-\frac{\gamma^2}{4},0)$ is well defined. We also want to mention that a similar result holds for $\tilde{U}(t)$ and the proof is almost the same.
\end{remark}
\begin{proof}
We adapt the arguments in \cite{DOZZ2} for the proof of this lemma. We introduce the notation
\begin{equation}
K_I(t) :=\int_I (x-t)^{\frac{\gamma^2}{4}}  x^{a} e^{\frac{\gamma}{2} X( x) } d x
\end{equation}
for a borel set $I \subseteq [0,1]$. Recall that we work with $-1-\frac{\gamma^2}{4}<a<-1-\frac{\gamma^2}{4}+a_0$ with $a_0$ small, hence $p<1+\frac{4}{\gamma^2}(a+1)<1$. We want to study the asymptotic of 
\begin{equation}
 \E[K_{[0,1]}(t)^{p}]-\E[K_{[0,1]}(0)^{p}] =: T_1+T_2,
\end{equation}
where we defined:
\begin{equation}
T_1: = \E[K_{[|t|,1]}(t)^{p}]-\E[K_{[0,1]}(0)^{p}], \quad T_2:= \E[K_{[0,1]}(t)^{p}]- \E[K_{[|t|,1]}(t)^{p}].
\end{equation}
$\Diamond $ First we consider $T_1$. The goal is to show that $T_1 = o(|t|^{1+a+\frac{\gamma^2}{4}})$. By interpolation,
\begin{align}\label{equation 5.4}
|T_1| \le &|p| \int_0^1 du \E[|K_{[|t|,1]}(t)-K_{[0,1]}(0)|(uK_{[|t|,1]}(t)+(1-u)K_{[0,1]}(0))^{p-1}] \nonumber\\
\le &|p| \E[|K_{[|t|,1]}(t)-K_{[0,1]}(0)|K_{[|t|,1]}(0)^{p-1}]\le |p|(A_1+A_2),
\end{align}
where $$A_1= \E[|K_{[|t|,1]}(t)-K_{[|t|,1]}(0)|K_{[|t|,1]}(0)^{p-1}]$$ and $$A_2 = \E[|K_{[|t|,1]}(0)-K_{[0,1]}(0)|K_{[|t|,1]}(0)^{p-1}].$$ We start by estimating $A_1$. Using the sub-additivity of the function $x\mapsto x^{\frac{\gamma^2}{4}}$,
\begin{align*}
A_1= & \E[|K_{[|t|,1]}(t)-K_{[|t|,1]}(0)|K_{[|t|,1]}(0)^{p-1}]\\ 
\le & |t|^{\frac{\gamma^2}{4}} \int_{|t|}^1  dx_1\, x_1^{a} \E[(\int_{|t|}^1 \frac{ x^{a+\frac{\gamma^2}{4}} }{|x-x_1|^{\frac{\gamma^2}{2}}}e^{\frac{\gamma}{2} X( x) } d x)^{p-1}]\\
\le &|t|^{\frac{\gamma^2}{4}} \int_{|t|}^{t_0}  dx_1\, x_1^{a} \E[(\int_{x_1}^1 x^{a-\frac{\gamma^2}{4}} e^{\frac{\gamma}{2} X( x) } d x)^{p-1}] +c|t|^{\frac{\gamma^2}{4}},
\end{align*}
where $t_0$ is a constant in $(0,1)$ to be fixed. Note that in this subsection we will use $c>0$ to denote a positive constant with the abuse of notation that it can be a different constant every time it appears. Here we now need to apply lemma \ref{lemma estimation GMC}. We check that the bounds of \eqref{bounds} on $p$ imply that $1 + a + (1-p) \frac{\gamma^2}{2} >0$. Therefore we are in the first case of lemma \ref{lemma estimation GMC} which implies there exists $\epsilon_1>0$ such that for all $x_1<\epsilon_1$:
\begin{equation}\label{equation 5.5}
\E[(\int_{x_1}^1 x^{a-\frac{\gamma^2}{4}} e^{\frac{\gamma}{2} X( x) } d x)^{p-1}] \le  c \,x_1^{\frac{1}{\gamma^2}(a+1)^2}.
\end{equation}
Taking $t_0 = \epsilon_1$ we obtain:
\begin{align}
A_1 \le &c |t|^{\frac{\gamma^2}{4}} \int_{|t|}^{\epsilon_1}  dx_1\, x_1^{a+\frac{1}{\gamma^2}(a+1)^2 }+c|t|^{\frac{\gamma^2}{4}}\nonumber\\
\le &c \,|t|^{1+\frac{\gamma^2}{4}+a+\frac{1}{\gamma^2}(a+1)^2}+c|t|^{\frac{\gamma^2}{4}}=o\big( |t|^{1+a+\frac{\gamma^2}{4}}\big).
\end{align}
On the other hand:
\begin{align}\label{equation 5.7}
A_2 &= \E[K_{[0,|t|]}(0)K_{[|t|,1]}(0)^{p-1}]\\ \nonumber
 &= \int_0^{|t|}  dx_1\, x_1^{a+\frac{\gamma^2}{4}} \E[(\int_{|t|}^1 \frac{ x^{a+\frac{\gamma^2}{4}} }{|x-x_1|^{\frac{\gamma^2}{2}}}e^{\frac{\gamma}{2} X( x) } d x)^{p-1}]\nonumber \\
&\le   \int_0^{|t|}  dx_1\, x_1^{a+\frac{\gamma^2}{4}} \E[(\int_{|t|}^1  x^{a-\frac{\gamma^2}{4}} e^{\frac{\gamma}{2} X( x) } d x)^{p-1}] \nonumber \\
&\overset{(\ref{equation 5.5})}{\le}  c |t|^{1+a+\frac{\gamma^2}{4}+\frac{1}{\gamma^2}(a+1)^2} = o(|t|^{1+a+\frac{\gamma^2}{4}}).
\end{align}
Hence we have shown that $T_1 = o(|t|^{1+a+\frac{\gamma^2}{4}})$.
\\~\\
$\Diamond$ Now we focus on $T_2$. The goal is to restrict $K$ to the complementary of  $[|t|^{1+h},|t|]$, with $h>0$ a constant to be fixed, and then on the two parts the GMC's are weakly correlated. The same computation as \eqref{equation 5.4} together with the technique we used for $T_1$ show that for $|t|$ sufficiently small: 
\begin{equation*}
\begin{split}
\E[K_{[0,1]}(t)^{p}]- \E[K_{[|t|^{1+h},|t|]^c}(t)^{p}] &\le |p| \E[K_{[|t|^{1+h},|t|]}(t) K_{[|t|,1]}(0)^{p-1}]\\
& \le c |t|^{\frac{\gamma^2}{4}} \int_{|t|^{1+h}}^{|t|}  dx_1\, x_1^{a+\frac{1}{\gamma^2}(a+1)^2  }\\
& \le  c |t|^{\frac{\gamma^2}{4}+(1+h) \left(1+a+\frac{1}{\gamma^2}(a+1)^2  \right)} 
\end{split}
\end{equation*}
By taking $h<-\frac{1+a}{1+a+\gamma^2}$, we have
\begin{equation}\label{condition h1}
\frac{\gamma^2}{4}+(1+h) (1+a+\frac{1}{\gamma^2}(a+1)^2)>1+a+\frac{\gamma^2}{4},
\end{equation}
hence
\begin{equation}
\E[K_{[0,1]}(t)^{p}]- \E[K_{[|t|^{1+h},|t|]^c}(t)^{p}] = o(|t|^{1+a+\frac{\gamma^2}{4}}).
\end{equation}
This means that it suffices to evaluate $\E[K_{[|t|^{1+h},|t|]^c}(t)^{p}]-\E[K_{[|t|,1]}(t)^{p}]$. We will use the radial decomposition of $X$ with the notations introduced in the first paragraph of section \ref{section estimation GMC},
\begin{equation}
K_1(t) := K_{[|t|,1]}(t) =  \frac{1}{2} \int_0^{2 \ln \frac{1}{|t| }} (e^{-s/2}-t)^{\frac{\gamma^2}{4}} \,  e^{ \frac{\gamma}{2}(B_s - s( \frac{\gamma}{4} +\frac{1}{\gamma}(a+1)) ) } \mu_{Y}(ds) ,
\end{equation}
\begin{equation}
K_2(t)  := K_{[0,|t|^{1+h}]}(t) =  \frac{1}{2} \int_{2(1+h)\ln \frac{1}{|t|}}^{\infty} (e^{-s/2}-t)^{\frac{\gamma^2}{4}} \,  e^{ \frac{\gamma}{2}(B_s - s( \frac{\gamma}{4} +\frac{1}{\gamma}(a+1)) ) } \mu_{Y}(ds) .
\end{equation}
From (\ref{equation correlation Y}), we deduce that for $s\le 2 \ln \frac{1}{|t|}$ and $s' \ge 2(1+h) \ln \frac{1}{|t|} $,
\begin{equation}
0 \le \E[Y(e^{-s/2})Y(e^{-s'/2})] = \ln \frac{1}{|1-e^{-(s'-s)/2}|} \le 2|t|^h,
\end{equation}
where we used the inequality $\ln\frac{1}{1-x} \le 2x$ for $x\in [0,\frac{1}{2}]$. Define the processes,
\begin{equation*}
P(e^{-s/2}) := Y(e^{-s/2})\mathbf{1}_{\{s\le 2 \ln \frac{1}{|t|}\}} + Y(e^{-s/2})\mathbf{1}_{\{s \ge 2(1+h) \ln \frac{1}{|t|}\}},
\end{equation*}
\begin{equation*}
\tilde{P}(e^{-s/2}) := Y(e^{-s/2})\mathbf{1}_{\{s\le 2 \ln \frac{1}{|t|}\}} + \tilde{Y}(e^{-s/2})\mathbf{1}_{\{s \ge 2(1+h) \ln \frac{1}{|t|}\}},
\end{equation*}
where $\tilde{Y}$ is a gaussian field independent from everything and has the same law as $Y$. Then we have the inequality over the covariance:
\begin{equation}
\E[\tilde{P}(e^{-s/2})\tilde{P}(e^{-s'/2})] \le \E[P(e^{-s/2})P(e^{-s'/2})] \le \E[\tilde{P}(e^{-s/2})\tilde{P}(e^{-s'/2})]+2|t|^h.
\end{equation}
The function $x\mapsto x^{p}$ is convex when $p\le 0$ and concave when $0<p<1 $. We will only work with the case $p \le 0$ since the case $ 0<p<1$ can be treated in the same way. By applying Kahane's inequality of Theorem \ref{theo convex inequality},
\begin{equation}
\E[(K_1(t) +\tilde{K}_2(t) )^p] \le \E[(K_1(t) +K_2(t) )^p] \le e^{\frac{\gamma^2}{4}(p^2-p)|t|^h}\E[(K_1(t) +\tilde{K}_2(t) )^p],
\end{equation}
where $\tilde{K}_2(t)  :=   \frac{1}{2} \int_{2(1+h)\ln \frac{1}{|t|}}^{\infty} (e^{-s/2}-t)^{\frac{\gamma^2}{4}} \,  e^{ \frac{\gamma}{2}(B_s - s( \frac{\gamma}{4} +\frac{1}{\gamma}(a+1)) ) } \mu_{\tilde{Y}}(ds) $. By the Markov property of Brownian motion and stationarity of $\mu_{\tilde{Y}}$, we have

\begin{equation}
\begin{split}
\tilde{K}_2(t)  :=  & \frac{1}{2} |t|^{(1+h)(1+a+\frac{\gamma^2}{4})+\frac{\gamma^2}{4}}e^{\frac{\gamma}{2}B_{2(1+h)\ln (1/|t|)}}\\
&\int_{0}^{\infty} (|t|^h e^{-s/2}+1)^{\frac{\gamma^2}{4}} \,  e^{ \frac{\gamma}{2}(\tilde{B}_s - s( \frac{\gamma}{4} +\frac{1}{\gamma}(a+1)) ) } \mu_{\tilde{Y}}(ds),
\end{split}
\end{equation}
with $\tilde{B}$ an independent Brownian motion. We denote
\begin{equation}
\sigma_t:= |t|^{(1+h)(1+a+\frac{\gamma^2}{4})+\frac{\gamma^2}{4}}e^{\frac{\gamma}{2}B_{2(1+h)\ln (1/|t|)}},\quad V: =\frac{1}{2}\int_{0}^{\infty}  e^{ \frac{\gamma}{2}(\tilde{B}_s - s( \frac{\gamma}{4} +\frac{1}{\gamma}(a+1)) ) } \mu_{\tilde{Y}}(ds),
\end{equation}
then:
\begin{align}
\E[(K_1(t) +(1+|t|^h)^{\frac{\gamma^2}{4}}\sigma_t V)^p] &\le \E[(K_1(t) +K_2(t) )^p]\\ \nonumber
 &\le e^{\frac{\gamma^2}{4}(p^2-p)|t|^h}\E[(K_1(t) +\sigma_t V )^p].
\end{align}
By the Williams path decomposition of Theorem \ref{theo Williams} we can write,
\begin{equation}
V =  e^{\frac{\gamma}{2}M}\frac{1}{2}\int_{-L_M}^{\infty}  e^{ \frac{\gamma}{2}\mathcal{B}^{\lambda}_s} \mu_{\tilde{Y}}(ds),
\end{equation}
where $\lambda =  \frac{\gamma}{4} +\frac{1}{\gamma}(a+1)$, $ M = \sup_{s \geqslant 0} (\tilde{B}_s -\lambda s) $ and $L_M$ is the last time $\left( \mathcal{B}^{\lambda}_{-s} \right)_{s\ge 0} $ hits $-M$. Recall that the law of $M$ is known, for $v\ge 1$,
\begin{equation}\label{equation 5.18}
\mathbb{P}( e^{\frac{\gamma}{2} M} >v ) = \frac{1}{v^{ \frac{4\lambda}{\gamma}}}.
\end{equation}
For simplicity, we introduce the notations:
\begin{equation}\label{equation rho}
\rho_A(\lambda) :=\frac{1}{2}\int_{-L_A}^{\infty}  e^{ \frac{\gamma}{2}\mathcal{B}^{\lambda}_s} \mu_{\tilde{Y}}(ds), \quad \rho(\lambda) :=\frac{1}{2}\int_{-\infty}^{\infty}  e^{ \frac{\gamma}{2}\mathcal{B}^{\lambda}_s} \mu_{\tilde{Y}}(ds).
\end{equation}
Now we discuss the lower and upper bound separately.
\\~\\
$\Diamond$ Lower bound:
Since we work with $p\le 0$,
\begin{small}
\begin{align*}
\E[(&K_1(t) +K_2(t) )^p]-\E[K_1(t)^p]\\
&\ge  \E[(K_1(t) +(1+|t|^h)^{\frac{\gamma^2}{4}}\sigma_t e^{\frac{\gamma}{2}M}\rho(\lambda))^p]-\E[K_1(t)^p]\\
&= \frac{4\lambda }{\gamma}\E\Big[ \int_1^{\infty} \frac{dv}{v^{ \frac{4\lambda }{\gamma}+1}}  \Big( ( K_1(t)+(1+|t|^h)^{\frac{\gamma^2}{4}}\sigma_t \rho(\lambda) v )^p  - K_1(t)^p \Big)\Big] \\
&= \frac{4\lambda }{\gamma} \E\Big[\int_{\frac{(1+|t|^h)^{\frac{\gamma^2}{4}}\sigma_t\rho(\lambda)}{K_1(t)}}^{\infty} \frac{du}{u^{\frac{4\lambda }{\gamma}+1}}((u+1)^p -1)  ((1+|t|^h)^{\frac{\gamma^2}{4}}\sigma_t \rho(\lambda) )^{\frac{4\lambda }{\gamma}}K_1(t)^{p- \frac{4\lambda }{\gamma}}\Big] \\
&\overset{(\ref{equation integral computation 1})}{\ge}\frac{4\lambda}{\gamma}\frac{\Gamma(-p+ \frac{4\lambda}{\gamma})\Gamma(- \frac{4\lambda}{\gamma})}{\Gamma(-p)} \mathbb{E}[ ((1+|t|^h)^{\frac{\gamma^2}{4}}\sigma_t \rho(\lambda) )^{\frac{4\lambda }{\gamma}}K_1(t)^{p- \frac{4\lambda }{\gamma}}].
\end{align*}
\end{small}
By the Girsanov theorem,
\begin{small}
\begin{align}
\mathbb{E}[ (&(1+|t|^h)^{\frac{\gamma^2}{4}}\sigma_t \rho(\lambda) )^{\frac{4\lambda }{\gamma}}K_1(t)^{p- \frac{4\lambda }{\gamma}}]\nonumber\\
&= \big(|t|(1+|t|^h)\big)^{1+a+\frac{\gamma^2}{4}}\mathbb{E}[ \rho(\lambda)^{ \frac{4\lambda}{\gamma}}] \\ \nonumber 
& \times \E\big[\big(\frac{1}{2}\int_{0 }^{2\ln \frac{1}{\vert t \vert} } (e^{-s/2}-t)^{\frac{\gamma^2}{4}} e^{ \frac{\gamma}{2}(B_s + s( \frac{\gamma}{4} +\frac{1}{\gamma}(a+1)) ) } \mu_{Y}(ds)\big)^{p-\frac{4\lambda}{\gamma}}\big] \nonumber
\\
&\underset{t\to 0_-}{\sim} |t|^{1+a+\frac{\gamma^2}{4}}\mathbb{E}[ \rho(\lambda)^{ \frac{4\lambda}{\gamma}}] M(\gamma,p-1-\frac{4}{\gamma^2}(a+1),-2-a-\frac{\gamma^2}{4},0).
\end{align}
\end{small}
This completes the proof for lower bound.
\\~\\
$\Diamond$ Upper bound: we start with an inequality:
\begin{align}
 \E[\big(&(K_1(t) +K_2(t) )^p]-\E[K_1(t)^p] \nonumber \\
  \le & \E[(K_1(t) +\sigma_t V )^p]-\E[K_1(t)^p]+(e^{\frac{\gamma^2}{4}(p^2-p)|t|^h}-1) \E[K_1(0)^p]\\
  =&  \E[(K_1(t) +\sigma_t V )^p]-\E[K_1(t)^p]+O(|t|^h).
\end{align}
To get rid of the big O term, we will need an $h$ such that
\begin{equation}
h>1+a+\frac{\gamma^2}{4}.
\end{equation}
Together with the condition \ref{condition h1}, we have 
\begin{equation}
1+a+\frac{\gamma^2}{4} < h <-  \frac{1+a}{1+a+\gamma^2}.
\end{equation}
There exists such an $h$ when $a$ is sufficiently close to $-1-\frac{\gamma^2}{4}$.

For $A>0$ fixed, since $p\le 0$ we have,
\begin{small}
\begin{align*}
 \E[(&K_1(t) +\sigma_t V )^p-K_1(t)^p] \le \E[\big((K_1(t) +\sigma_t V )^p-K_1(t)^p\big) \mathbf{1}_{\{M>A\}}]\\
 & \le \E[\big((K_1(t) +\sigma_t\, e^{\frac{\gamma}{2}M} \rho_A(\lambda)  )^p-K_1(t)^p\big) \mathbf{1}_{\{M>A\}}]\\
 &\overset{(\ref{equation 5.18})}{=}\frac{4\lambda }{\gamma} \E \Big[\int_{\frac{e^{\gamma A/2}\sigma_t\rho_A(\lambda) }{K_1(t)}}^{\infty} \frac{du}{u^{\frac{4\lambda }{\gamma}+1}}((u+1)^p -1) (\sigma_t \rho_A(\lambda) )^{\frac{4\lambda }{\gamma}}K_1(t)^{p- \frac{4\lambda }{\gamma}}\Big] \\
 &\overset{\text{Girsanov}}{=} \frac{4\lambda }{\gamma}|t|^{1+a+\frac{\gamma^2}{4}} \E \Big[\int_{\frac{e^{\gamma A/2}\hat{\sigma}_t\rho_A(\lambda) }{\hat{K}_1(t)}}^{\infty} \frac{du}{u^{\frac{4\lambda }{\gamma}+1}}((u+1)^p -1) \rho_A(\lambda)^{\frac{4\lambda }{\gamma}}\hat{K}_1(t)^{p- \frac{4\lambda }{\gamma}}\Big],
\end{align*}
\end{small}
where
\begin{small}
\begin{align*}
\hat{K}_1(t)&= \E\big[\big(\frac{1}{2}\int_{0 }^{2\ln \frac{1}{\vert t \vert} } (e^{-s/2}-t)^{\frac{\gamma^2}{4}} e^{ \frac{\gamma}{2}(B_s + s( \frac{\gamma}{4} +\frac{1}{\gamma}(a+1)) ) } \mu_{Y}(ds)\big)^{p-\frac{4\lambda}{\gamma}}\big]\\
&\overset{t\to 0_-}{\sim} M(\gamma,p-1-\frac{4}{\gamma^2}(a+1),-2-a-\frac{\gamma^2}{4},0),
\end{align*}
\end{small}

and for $a<-1-\frac{h \gamma^2}{4(1+h)}$,
\begin{equation*}
\hat{\sigma_t} = |t|^{-(1+h)(1+a+\frac{\gamma^2}{4}) + \frac{\gamma^2}{4}}e^{\frac{\gamma}{2}B_{2(1+h)\ln(1/|t|)}}\overset{t\to 0_-} {\longrightarrow}0 \quad \text{a.s.}
\end{equation*}

Hence $ \E[(K_1(t) +\sigma_t V )^p-K_1(t)^p]$ is smaller than a term equivalent to:
\begin{small}
$$\frac{4\lambda}{\gamma}\frac{\Gamma(-p+ \frac{4\lambda}{\gamma})\Gamma(- \frac{4\lambda}{\gamma})}{\Gamma(-p)}|t|^{1+a+\frac{\gamma^2}{4}}\mathbb{E}[ \rho_A(\lambda)^{ \frac{4\lambda}{\gamma}}] M(\gamma,p-1-\frac{4}{\gamma^2}(a+1),-2-a-\frac{\gamma^2}{4},0).$$
\end{small}
We can conclude by sending $A$ to $\infty$.
\end{proof}

\subsection{Computation of the reflection coefficient}\label{sec_reflection}
The goal of this subsection is to prove the tail expansion result for GMC given by Proposition \ref{proposition reflection}. In the first step we give a proof of the tail expansion \eqref{tail_result1} where the coefficient $\overline{R}_1^{\partial}$ is expressed in terms of the processes $Y$ and $\mathcal{B}_s^{\alpha}$ as defined in the section \ref{section tail}. The proof is almost the same as in \cite{DOZZ2}. In the second step we provide the exact value \eqref{tail_result2} for $\overline{R}_1^{\partial}$ by using Theorem \ref{main_result}. Before proving the proposition, we provide a useful lemma. The proof can be found in \cite{DOZZ2} (see lemma 2.8). 
\begin{lemma}\label{lemme_dozz}
Let $\alpha \in (\frac{\gamma}{2},Q)$ with $Q=\frac{\gamma}{2} +\frac{2}{\gamma}$, then for $p<\frac{4}{\gamma^2}$ and all non trivial interval $I \subseteq \mathbb{R}$:
\begin{equation}
\E[(\frac{1}{2}\int_I  e^{ \frac{\gamma}{2}\mathcal{B}^{\frac{Q-\alpha}{2}}_s} e^{\frac{\gamma}{2} Y(e^{-s/2})  } ds)^p]< \infty.
\end{equation}
\end{lemma}
This lemma tells us that the additional term $e^{ \frac{\gamma}{2}\mathcal{B}^{\frac{Q-\alpha}{2}}_s} $ behaves nicely and the bound on $p$ is the same as in the case of GMC moments.
\begin{proof}[Proof of Proposition \ref{proposition reflection}]
Using the decomposition $X(e^{-s/2}) = B_s + Y(e^{-s/2})$ we have,
\begin{align*}
I^{\partial}_{1,\eta}( \alpha ) &= \int_0^{\eta} x^{-\frac{\gamma \alpha}{2}} e^{\frac{\gamma}{2} X(x)} dx=\frac{1}{2} \int_{-2\ln \eta}^{\infty} e^{ \frac{\gamma}{2}(B_s-s(\frac{\gamma}{4}+\frac{1}{\gamma}-\frac{\alpha}{2})) } e^{\frac{\gamma}{2} Y(e^{-s/2})  } ds\\
&\overset{\text{Theorem \ref{theo Williams}}}{=}e^{\frac{\gamma}{2}M}\frac{1}{2}\int_{-2\ln \eta-L_M}^{\infty}  e^{ \frac{\gamma}{2}\mathcal{B}^{\frac{Q-\alpha}{2}}_s} e^{\frac{\gamma}{2} Y(e^{-s/2})  } ds,
\end{align*}
where  $ M = \sup_{s \geqslant 0} (B_s -\frac{Q-\alpha}{2} s) $ and $L_M$ is the last time $\left( \mathcal{B}^{\frac{Q-\alpha}{2}}_{-s} \right)_{s \ge 0}$ hits $-M$. The law of $M$ is given by:
\begin{equation}
\mathbb{P}( e^{\frac{\gamma}{2} M} >v ) = \frac{1}{v^{ \frac{2(Q-\alpha)}{\gamma}}} \quad (v\ge 1).
\end{equation}
We denote
\begin{align*}
&\rho_A(\frac{Q-\alpha}{2}) = \frac{1}{2}\int_{-L_A}^{\infty}  e^{ \frac{\gamma}{2}\mathcal{B}^{\frac{Q-\alpha}{2}}_s} e^{\frac{\gamma}{2} Y(e^{-s/2})  } ds,\\
&\rho(\frac{Q-\alpha}{2}) = \frac{1}{2}\int_{-\infty}^{\infty}  e^{ \frac{\gamma}{2}\mathcal{B}^{\frac{Q-\alpha}{2}}_s} e^{\frac{\gamma}{2} Y(e^{-s/2})  } ds,
\end{align*}
and study the upper and lower bounds for $\P(I^{\partial}_{1,\eta}( \alpha )>u)$.\\
$\Diamond$ Upper bound: 
$$\P(I^{\partial}_{1,\eta}( \alpha )>u) \le \P(e^{\frac{\gamma}{2}M}\rho(\frac{Q-\alpha}{2})>u)  \le  \frac{\E[\rho(\frac{Q-\alpha}{2})^{\frac{2(Q-\alpha)}{\gamma}}]}{u^{\frac{2(Q-\alpha)}{\gamma}}}=\frac{ \overline{R}_1^{\partial}(\alpha)}{u^{ \frac{2}{\gamma}(Q - \alpha) }}.$$
$\Diamond$ Lower bound: we first show that the tail behavior is concentrated at $x = 0$ and that the value of $\eta$ does not matter. Consider $h,\epsilon>0$ sufficiently small,
\begin{align}
\P(I^{\partial}_{1,1}( \alpha )>u+u^{1-h})&-\P(I^{\partial}_{1,\eta}( \alpha )>u)\le \P(\int_{\eta}^1 x^{-\frac{\gamma \alpha}{2}} e^{\frac{\gamma}{2} X(x)} dx >u^{1-h})\nonumber \\
& \le \frac{\E[(\int_{\eta}^1 x^{-\frac{\gamma \alpha}{2}} e^{\frac{\gamma}{2} X(x)} dx)^{\frac{4}{\gamma^2}-\epsilon}]}{u^{(1-h)(\frac{4}{\gamma^2}-\epsilon)}}
 = O_{u\to \infty}(\frac{1}{u^{\frac{2(Q-\alpha)}{\gamma}+\nu}}), \label{equation 442}
\end{align}
 where $\nu>0$ can be any constant that satisfies $\nu \le (1-h)(\frac{4}{\gamma^2}-\epsilon) - \frac{2(Q-\alpha)}{\gamma^2}$. Thus it suffices to study the tail behavior of $I^{\partial}_{1,1}( \alpha )$.
Take $A=\frac{2\nu}{\gamma}\ln u$, 
\begin{align*}
&\quad \P(I^{\partial}_{1,1}( \alpha )>u) \\
&\ge \P(e^{\frac{\gamma}{2}M}\rho_A(\frac{Q-\alpha}{2})>u, M>A)\\ 
 &=  \P \left(e^{\frac{\gamma}{2}M}> \max \left\{ \frac{u}{\rho_A(\frac{Q-\alpha}{2})}, e^{\frac{\gamma }{2}A}    \right\} \right)  \\
 &=  \E \left[  \min \left\{ \frac{\rho_A(\frac{Q-\alpha}{2})}{u}, \frac{1}{u^{\nu}}  \right\}^{\frac{2(Q-\alpha)}{\gamma}}   \right]     \\
 & \ge   u^{-\frac{2(Q-\alpha)}{\gamma}} \left( \E [  \rho_A(\frac{Q-\alpha}{2})^{\frac{2(Q-\alpha)}{\gamma}}  ]-\E [  \rho_A(\frac{Q-\alpha}{2})^{\frac{2(Q-\alpha)}{\gamma}} \mathbf{1}_{\rho_A(\frac{Q-\alpha}{2}) > u^{1-\nu}} ] \right)   
\end{align*}
 Take $h'>1$ a constant such that $h'\frac{2(Q-\alpha)}{\gamma} < \frac{4}{\gamma^2}$, by H\"older's inequality and Markov inequality:
\begin{align*}
&E [  \rho_A(\frac{Q-\alpha}{2})^{\frac{2(Q-\alpha)}{\gamma}} \mathbf{1}_{\rho_A(\frac{Q-\alpha}{2}) > u^{1-\nu}} ] \\
\le &\E[\rho_A(\frac{Q-\alpha}{2})^{h'\frac{2(Q-\alpha)}{\gamma}}]^{\frac{1}{h'}} \P(\rho_A(\frac{Q-\alpha}{2})^{\frac{2(Q-\alpha)}{\gamma}} > u^{1-\nu})^{\frac{h'-1}{h'}}\\
\le & \E[\rho_A(\frac{Q-\alpha}{2})^{h'\frac{2(Q-\alpha)}{\gamma}}] u^{-(1-\nu)(h'-1)} = O(u^{-(1-\nu)(h'-1)}).
\end{align*}
We impose additionally that $\nu$ satisfies $\nu < (1-\nu)(h'-1)$, then
\begin{equation}
\P(I^{\partial}_{1,1}( \alpha )>u) \ge u^{-\frac{2(Q-\alpha)}{\gamma}} \E [  \rho_A(\frac{Q-\alpha}{2})^{\frac{2(Q-\alpha)}{\gamma}}  ] + O(u^{-\frac{2(Q-\alpha)}{\gamma}-\nu}).
\end{equation}

We claim that for $u>1$ and for some $c>0$,
\begin{equation}\label{equation 4.42}
\E[\rho(\frac{Q-\alpha}{2})^{\frac{2(Q-\alpha)}{\gamma}}]-\E[\rho_A(\frac{Q-\alpha}{2})^{\frac{2(Q-\alpha)}{\gamma}}] \le c u^{-\nu}.
\end{equation} 
This shows that:
\begin{equation}
\P(I^{\partial}_{1,1}( \alpha )>u) =\frac{\overline{R}_1^{\partial}(\alpha)}{u^{\frac{2(Q-\alpha)}{\gamma}}} + O(\frac{1}{u^{\frac{2(Q-\alpha)}{\gamma}+\nu}}) .
\end{equation}
By applying the tail result to \eqref{equation 442} we deduce,
\begin{equation}
\P(I^{\partial}_{1,\eta}( \alpha )>u) =\frac{\overline{R}_1^{\partial}(\alpha)}{u^{\frac{2(Q-\alpha)}{\gamma}}} + O(\frac{1}{u^{\frac{2(Q-\alpha)}{\gamma}+\min(\nu,h)}}) ,
\end{equation}
which finishes the proof for the first part.
For the second part let $\epsilon >0$, the value of $ \overline{R}_1^{\partial}(\alpha)$ is then determined by the following limit, with $p= \frac{2(Q-\alpha)}{\gamma}$,
\begin{equation}
 \lim_{\epsilon \rightarrow 0} \epsilon \mathbb{E} [ I_{1,1}^{\partial}(\alpha)^{p- \epsilon} ] = p \overline{R}_1^{\partial}(\alpha).
\end{equation}
With our Theorem \ref{main_result} we can compute this limit and get:
\begin{footnotesize}
\begin{align*}
p\overline{R}_1^{\partial}(\alpha) &=  \frac{(2 \pi)^p (\frac{2}{\gamma})^{p\frac{\gamma^2}{4}}    
\Gamma_{\frac{\gamma}{2}} ( \frac{2}{\gamma} -p \frac{\gamma}{2} )
\Gamma_{\frac{\gamma}{2}}( \frac{2}{\gamma} -(p-1) \frac{\gamma}{2}  ) \Gamma_{\frac{\gamma}{2}}(
\frac{4}{\gamma} -\alpha-(p-2) \frac{\gamma}{2}  ) }{ \Gamma( 1  -  \frac{\gamma^2}{4} ) ^p    \Gamma_{\frac{\gamma}{2}}(\frac{2}{\gamma}) \Gamma_{\frac{\gamma}{2}}( \frac{2}{\gamma}-\alpha+
\frac{\gamma}{2}  ) \Gamma_{\frac{\gamma}{2}}( \frac{2}{\gamma} + \frac{\gamma}{2}  )
\Gamma_{\frac{\gamma}{2}}( \frac{4}{\gamma}-\alpha -(2p-2) \frac{\gamma}{2}  )  } \lim_{\epsilon \rightarrow 0} \epsilon \Gamma_{\frac{\gamma}{2}}(\frac{\gamma \epsilon }{2}) \\
&= \frac{(2 \pi)^p (\frac{2}{\gamma})^{p\frac{\gamma^2}{4}}\Gamma_{\frac{\gamma}{2}}( \alpha- \frac{\gamma}{2}  ) }{ \Gamma( 1  -  \frac{\gamma^2}{4} ) ^p 
    \Gamma_{\frac{\gamma}{2}}(\frac{2}{\gamma}) \Gamma_{\frac{\gamma}{2}}( Q-\alpha )    } \frac{1}{\sqrt{2\pi}} (\frac{\gamma}{2})^{-\frac{1}{2}}\Gamma_{\frac{\gamma}{2}}(\frac{2}{\gamma})\\
&= \frac{1}{\sqrt{\gamma \pi}  } \frac{ (2 \pi)^{ \frac{2}{\gamma}(Q -\alpha ) } (\frac{2}{\gamma})^{ \frac{\gamma}{2}(Q -\alpha )  }  }{ \Gamma(1 -\frac{\gamma^2}{4}  )^{ \frac{2}{\gamma}(Q -\alpha ) } } \frac{ \Gamma_{\frac{\gamma}{2}}(\alpha - \frac{\gamma}{2}  )}{\Gamma_{\frac{\gamma}{2}}(Q- \alpha )} .
\end{align*}
\end{footnotesize}
It remains to show \eqref{equation 4.42}. By \eqref{equation 4.3} of the Williams decomposition theorem of appendix \ref{sec_th}, the process $\hat{\mathcal{B}}_s^{\frac{Q-\alpha}{2}} $ defined for $s\le 0$ by $$\hat{\mathcal{B}}_s^{\frac{Q-\alpha}{2}} = \mathcal{B}_{s-L_{\frac{2\nu}{\gamma}\ln u}}^{\frac{Q-\alpha}{2}}+\frac{2\nu}{\gamma}\ln u$$ is independent from everything and has the same law as $(\mathcal{B}^{\frac{Q-\alpha}{2}}_s)_{s\le 0}$. We can then write,
\begin{equation}
\rho(\frac{Q-\alpha}{2})= A_1+u^{-\nu} A_2,
\end{equation}
where:
\begin{align}
&A_1 = \frac{1}{2}\int_{-L_{\frac{2\nu}{\gamma}\ln u}}^{\infty}  e^{ \frac{\gamma}{2}\mathcal{B}^{\frac{Q-\alpha}{2}}_s} e^{\frac{\gamma}{2} Y(e^{-s/2})  } ds,\\ \nonumber
  &A_2 = \frac{1}{2}\int_{-\infty}^{0}  e^{ \frac{\gamma}{2}\hat{\mathcal{B}}^{\frac{Q-\alpha}{2}}_s} e^{\frac{\gamma}{2} Y(e^{-s/2})  } ds.
\end{align}
By interpolation (see \eqref{equation 5.4} for example),
\begin{align*}
\E[(&A_1+u^{-\nu} A_2)^{\frac{2(Q-\alpha)}{\gamma}}-A_1^{\frac{2(Q-\alpha)}{\gamma}}]\\ 
&\le \frac{2(Q-\alpha)}{\gamma}u^{-\nu}\E[A_2 \max\{\rho(\frac{Q-\alpha}{2})^{\frac{2(Q-\alpha)}{\gamma}-1},A_1^{\frac{2(Q-\alpha)}{\gamma}-1}\}].
\end{align*}
If $\frac{2(Q-\alpha)}{\gamma}\le 1$,
\begin{align*}
\E[(A_1&+u^{-\nu} A_2)^{\frac{2(Q-\alpha)}{\gamma}}-A_1^{\frac{2(Q-\alpha)}{\gamma}}] \le  u^{-\nu}\E[A_2 A_1^{\frac{2(Q-\alpha)}{\gamma}-1}]\\
 &\overset{\text{H\"older}}{\le} u^{-\nu}\E[A_2^{p}]^{1/p} \E[A_1^{\frac{p}{p-1}(\frac{2(Q-\alpha)}{\gamma}-1)}] ^{(p-1)/p}<cu^{-\nu},
\end{align*}
where $1<p<\frac{4}{\gamma^2}$ to ensure that $\E[A_2^p]$ is finite, and we know that $$A_1 \ge  \frac{1}{2}\int_{0}^{\infty}  e^{ \frac{\gamma}{2}\mathcal{B}^{\frac{Q-\alpha}{2}}_s} e^{\frac{\gamma}{2} Y(e^{-s/2})  } ds$$ has negative moments. On the other hand, if $\frac{2(Q-\alpha)}{\gamma}>1$, then:
\begin{align*}
\E[(A_1+u^{-\nu} A_2)^{\frac{2(Q-\alpha)}{\gamma}}-A_1^{\frac{2(Q-\alpha)}{\gamma}}] \le & \frac{2(Q-\alpha)}{\gamma} u^{-\nu}\E[\rho(\frac{Q-\alpha}{2})^{\frac{2(Q-\alpha)}{\gamma}}] < cu^{-\nu}.
\end{align*}
This last upper bound comes from the fact that the moment of $\rho(\frac{Q-\alpha}{2})$ is finite thanks to Lemma \ref{lemme_dozz} and since $\frac{2(Q-\alpha)}{\gamma}<\frac{4}{\gamma^2}$.
\end{proof}

\section{Special functions}\label{sec_special}

Lastly we include here a detailed discussion on hypergeometric functions and on the special functions $\Gamma_{\frac{\gamma}{2}}$ and $G$ that we have used in our paper. First, let us discuss the theory of hypergeometric equations and the so-called connection formulas between the different bases of their solutions. For $A>0$ let $\Gamma(A) = \int_0^{\infty} t^{A-1} e^{-t} dt $ denote the standard Gamma function and let $(A)_n : = \frac{\Gamma(A +n)}{\Gamma(A)}$. For $A, B, C,$ and $x$ real numbers we define the hypergeometric function $F$ by:
\begin{equation}
F(A,B,C,x) :=  \sum_{n=0}^{\infty} \frac{(A)_n (B)_n}{n! (C)_n} x^n.
\end{equation}
This function can be used to solve the following hypergeometric equation:
\begin{equation}
( t (1-t) \frac{d^2}{d t^2} + ( C - (A +B +1)t) \frac{d}{dt} - AB) U(t) =0.
\end{equation}
For our purposes we will always work with the parameter $t \in (- \infty,0)$ and we can give the following two bases of solutions, under the assumption that $C$ and $A - B$ are not integers,
\begin{align*}
U(t) &= C_1 F(A,B,C,t) \\ 
&+ C_2 |t|^{1 -C} F(1 + A-C, 1 +B - C, 2 -C, t) \\
&= D_1 |t|^{-A}F(A,1+A-C,1+A-B,t^{-1})\\ 
& + D_2 |t|^{-B} F(B, 1 +B - C, 1 +B - A, t^{-1}),
\end{align*}
where the first expression is an expansion in power of $\vert t \vert$ and the second is an expansion in powers of $ \vert t \vert^{-1}$. For each basis we have two real constants that parametrize the solution space, $C_1, C_2$ and $D_1, D_2$. We thus expect to have an explicit change of basis formula that will give a link between $C_1, C_2$ and $D_1, D_2$. This is precisely what give the so-called connection formulas:
\begin{equation}\label{hpy1}
\begin{pmatrix}
C_1  \\
C_2
\end{pmatrix} 
= 
 \begin{pmatrix}
\frac{\Gamma(1 -C ) \Gamma( A - B +1  )  }{\Gamma(A - C +1 )\Gamma(1- B   )  } & \frac{\Gamma(1 -C) \Gamma( B- A +1 )  }{\Gamma(B - C +1 )\Gamma( 1- A )   } \\
 \frac{\Gamma(C-1 ) \Gamma( A - B +1  )  }{\Gamma( A  )  \Gamma( C - B  )} & \frac{\Gamma(C-1 )  \Gamma(B - A +1 ) }{ \Gamma(  B ) \Gamma( C - A )  }
\end{pmatrix} 
\begin{pmatrix}
D_1  \\
D_2 
\end{pmatrix}.
\end{equation}
This relation comes from the theory of hypergeometric equations and we will extensively use it to deduce our shift equations. We will apply it for both hypergeometric equations of Proposition \ref{BPZ}.

We will now provide some explanations on the function $\Gamma_{\frac{\gamma}{2}}(x)$ that we have introduced as well as its connection with the so-called $G$ Barnes' function. Our function $\Gamma_{\frac{\gamma}{2}}(x)$ is equal to the function $\Gamma_b(x)$ defined in the appendix of \cite{nakayama} with $b = \frac{\gamma}{2}$. \footnote{In \cite{Ostro_review} Ostrovsky uses a slightly different special function $\Gamma_2(x \vert \tau) $, the relation with our $\Gamma_{\frac{\gamma}{2}}(x)$ is:  
\begin{equation*}
\Gamma_{\frac{\gamma}{2}}(x) =  (\frac{2}{\gamma})^{\frac{1}{2}(x-\frac{Q}{2})^2} \frac{\Gamma_2( \frac{2x}{\gamma} \vert \tau )}{\Gamma_{2}(\frac{Q}{\gamma}|\tau)}.
\end{equation*}} For all $\gamma \in (0,2) $ and for $x >0$, $\Gamma_{\frac{\gamma}{2}}(x)$ is defined by the integral formula written in Theorem \ref{main_result},
\begin{equation}
\ln \Gamma_{\frac{\gamma}{2}}(x) = \int_0^{\infty} \frac{dt}{t} \left[ \frac{ e^{-xt} -e^{- \frac{Qt}{2}}   }{(1 - e^{- \frac{\gamma t}{2}})(1 - e^{- \frac{2t}{\gamma}})} - \frac{( \frac{Q}{2} -x)^2 }{2}e^{-t} + \frac{ x -\frac{Q}{2}  }{t} \right],
\end{equation} 
where we have $Q = \frac{\gamma}{2} +\frac{2}{\gamma}$. Since the function $\Gamma_{\frac{\gamma}{2}}(x)$ is continuous it is completely determined by the following two shift equations,
\begin{align}
\frac{\Gamma_{\frac{\gamma}{2}}(x)}{\Gamma_{\frac{\gamma}{2}}(x + \frac{\gamma}{2}) }&= \frac{1}{\sqrt{2 \pi}}
\Gamma(\frac{\gamma x}{2}) ( \frac{\gamma}{2} )^{ -\frac{\gamma x}{2} + \frac{1}{2}
}, \\
\frac{\Gamma_{\frac{\gamma}{2}}(x)}{\Gamma_{\frac{\gamma}{2}}(x + \frac{2}{\gamma}) }&= \frac{1}{\sqrt{2 \pi}} \Gamma(\frac{2
x}{\gamma}) ( \frac{\gamma}{2} )^{ \frac{2 x}{\gamma} - \frac{1}{2} },
\end{align}
and by its value in $\frac{Q}{2}$, $\Gamma_{\frac{\gamma}{2}}(\frac{Q}{2} ) =1$. We mention that $\Gamma_{\frac{\gamma}{2}}(x)$ is an analytic function of $x$. 
In the case where $\gamma =2$ the function $\Gamma_{\frac{\gamma}{2}}(x)$ reduces to,
\begin{equation}
\Gamma_{1}(x) = (2 \pi)^{\frac{x}{2} - \frac{1}{2}}  G(x)^{-1},
\end{equation}
where $G(x)$ is the so-called Barnes G function. This function is useful when we study the limit $\gamma  \rightarrow 2$ in section \ref{sec_app}. Finally in our Corollary \ref{law_ostro} we have used a special $\beta_{2,2}$ distribution defined in \cite{Ostro_review}. Here we recall the definition:
\begin{definition} [Existence theorem] $\\$
The distribution $-\ln \beta_{2,2}(a_1,a_2;b_0,b_1,b_2)$ is infinitely divisible on $[0,\infty)$ and has the L\'evy-Khintchine decomposition for $\operatorname{Re} (p) >-b_0$:
\begin{small}
\begin{align}
&\E[\exp(p \ln \beta_{2,2}(a_1, a_2;b_0,b_1,b_2))]\\ \nonumber
 & \quad= \exp \Big(\int_0^{\infty} (e^{-pt}-1)e^{-b_0t} \frac{(1-e^{-b_1t})(1-e^{-b_2t})}{(1-e^{-a_1t})(1-e^{-a_2t})}  \frac{dt}{t}\Big).
\end{align}
\end{small}
Furthermore, the distribution $\ln \beta_{2,2}(a_1,a_2;b_0,b_1,b_2)$ is absolutely continuous with respect to the Lebesgue measure.
\end{definition}
We only work with the case $(a_1,a_2)= (1,\frac{4}{\gamma^2})$. Then  $\beta_{2,2}(1,\frac{4}{\gamma^2};b_0,b_1,b_2)$ depends on 4 parameters $\gamma, b_0, b_1, b_2$ and its real moments $p>-b_0$ are given by the formula:
\begin{small}
\begin{align}
&\mathbb{E}[ \beta_{2,2}(1, \frac{4}{\gamma^2}; b_0, b_1, b_2)^p]\\  \nonumber
&\quad = \frac{\Gamma_{\frac{\gamma}{2}}(\frac{\gamma}{2}(p + b_0)) \Gamma_{\frac{\gamma}{2}}( \frac{\gamma}{2}(b_0 + b_1)) \Gamma_{\frac{\gamma}{2}}(\frac{\gamma}{2}(b_0 + b_2)) \Gamma_{\frac{\gamma}{2}}(\frac{\gamma}{2}(p + b_0 + b_1 +b_2))}{\Gamma_{\frac{\gamma}{2}}(\frac{\gamma}{2}b_0)\Gamma_{\frac{\gamma}{2}}(\frac{\gamma}{2}(p + b_0 + b_1))\Gamma_{\frac{\gamma}{2}}(\frac{\gamma}{2}(p + b_0 +b_2)) \Gamma_{\frac{\gamma}{2}}( \frac{\gamma}{2}(b_0 +b_1 +b_2))}.
\end{align}
\end{small}
Of course we have $\gamma \in (0,2)$ and the real numbers $p, b_0, b_1, b_2$ must be chosen so that the arguments of all the $\Gamma_{\frac{\gamma}{2}}$ are positive.
We conclude this section with a few computations that we need that also involve hypergeometric functions.
\begin{lemma}
For $p< 0$ and $-1<a<0$ or for $0<p<1$ and $-1<a<-p$ we have the identity:
\begin{small}
\begin{equation}\label{equation integral computation 1}
\int_0^{\infty} ( ( u+1)^{ p } - 1) u^{a-1} du=\frac{\Gamma(a)\Gamma(-a-p)}{\Gamma(-p)}.
\end{equation}
\end{small}
\end{lemma}
\begin{proof}
Denote by $(x)_n:= x(x+1)\dots (x+n-1)$.
\begin{small}
\begin{align*}
\int_0^{\infty} ( ( u+1)^{ p } &- 1) u^{a-1} du = \sum_{n=0}^{\infty} \frac{(-1)^n}{n!} (-p)_n \frac{1}{n +a } - \sum_{n=0}^{\infty} \frac{(-1)^n}{n!} (-p)_n \frac{1}{ a  +p-n } \\
=&\frac{1}{a}\sum_{n=0}^{\infty} \frac{(-1)^n}{n!} \frac{(-p)_n (a)_n}{(a +1)_n} - \frac{1}{a + p}\sum_{n=0}^{\infty} \frac{(-1)^n}{n!} \frac{(-p)_n(-a - p)_n}{ (-a - p+1 )_n } \\
=&\frac{1}{a}F(-p,a,a+1,-1)-\frac{1}{a+p}F(-p,-a-p,-a-p+1,-1)\\
=&\frac{\Gamma(a)\Gamma(-a-p)}{\Gamma(-p)},
\end{align*}
\end{small}
where in the last line we used the formula, for suitable $\overline{a}, \overline{b} \in \mathbb{R}$,
\begin{small}
\begin{equation*}
\bar{b}F(\bar{a}+\bar{b},\bar{a},\bar{a}+1,-1)+\bar{a}F(\bar{a}+\bar{b},\bar{b},\bar{b}+1,-1)=\frac{\Gamma(\bar{a}+1)\Gamma(\bar{b}+1)}{\Gamma(\bar{a}+\bar{b})}.
\end{equation*}
\end{small}
\end{proof}
\begin{lemma}
For $0<a<1-\frac{\gamma^2}{4}$ we have:
\begin{small}
\begin{equation}\label{equation integral computation 2}
\frac{\gamma^2}{4}\int_0^{\infty}  (y+1)^{ \frac{\gamma^2}{4}-1}y^{a-1}dy = (a+\frac{\gamma^2}{4})\frac{\Gamma(a)\Gamma(-a-\frac{\gamma^2}{4})}{\Gamma(-\frac{\gamma^2}{4})}.
\end{equation}
\end{small}
\end{lemma}
\begin{proof}
By the previous lemma,
\begin{small} 
\begin{equation*}
\int_0^{\infty} ( ( y+z)^{ \frac{\gamma^2}{4} } - 1) y^{a-1} dy=z^{a+\frac{\gamma^2}{4}}\frac{\Gamma(a)\Gamma(-a-\frac{\gamma^2}{4})}{\Gamma(-\frac{\gamma^2}{4})}.
\end{equation*}
\end{small}
We take the derivative in $z$ in the above equation and evaluate it at $z=1$ to get:
\begin{small}
\begin{equation*}
\frac{\gamma^2}{4}\int_0^{\infty}  (y+1)^{ \frac{\gamma^2}{4}-1}y^{a-1}dy = (a+\frac{\gamma^2}{4})\frac{\Gamma(a)\Gamma(-a-\frac{\gamma^2}{4})}{\Gamma(-\frac{\gamma^2}{4})}.
\end{equation*}
\end{small}
\end{proof}

\section*{Acknowledgements}
We are grateful to R\'emi Rhodes and Vincent Vargas for making us discover Gaussian multiplicative chaos and Liouville conformal field theory. We also very warmly thank Yan Fyodorov, Jon Keating, and Pierre Le Doussal for many fruitful discussions. Lastly we would like to thank the anonymous referee for his numerous remarks and suggestions that helped improve this paper.

\end{document}